\documentclass[a4paper,reqno,10pt]{amsart}
\usepackage{amssymb,amsmath,mathrsfs,bm}
\usepackage{paralist}
\usepackage[usenames]{color}
\usepackage[all]{xy}
\usepackage{colortbl}
\usepackage{arydshln}
\usepackage{multirow}
\usepackage {diagbox}
\usepackage{url}
\makeatletter

\@addtoreset{equation}{section}
\makeatother
\newtheorem{theorem}{Theorem}[section]
\newtheorem*{theorem*}{Main Theorem}
\newtheorem{corollary}[theorem]{Corollary}
\newtheorem{lemma}[theorem]{Lemma}
\newtheorem{proposition}[theorem]{Proposition}

\newtheorem{definition-proposition}[theorem]{Definition-Proposition}
\theoremstyle{definition}
\newtheorem{definition}[theorem]{Definition}
\newtheorem{remark}[theorem]{Remark}

\newtheorem*{question*}{Question}
\newtheorem{question**}{Question}

\newtheorem*{conjecture*}{Conjecture}
\newtheorem{example}[theorem]{Example}

\newtheorem{claim}{Claim}
\newtheorem*{claim*}{Claim}
\newtheorem{mainthm}{Theorem}

\def\Mor{\operatorname{Mor}}\def\Ob{\operatorname{Ob}}

\def\Hom{\operatorname{Hom}}
\def\mod{\operatorname{\mathsf{mod}}}

\def\Ab{\mathsf{Ab}}
\def\add{\operatorname{\mathsf{add}}}

\newcommand{\op}{\mathsf{op}}
\newcommand{\A}{\mathcal{A}}
\newcommand{\C}{\mathcal{C}}\newcommand{\D}{\mathcal{D}}

\renewcommand{\H}{\mathcal{H}}

\renewcommand{\L}{\mathcal{L}}\newcommand{\M}{\mathcal{M}}\newcommand{\N}{\mathcal{N}}
\newcommand{\R}{\mathcal{R}}
\renewcommand{\S}{\mathcal{S}}
\newcommand{\U}{\mathcal{U}}\newcommand{\V}{\mathcal{V}}
\newcommand{\W}{\mathcal{W}}\newcommand{\X}{\mathcal{X}}
\newcommand{\End}{\operatorname{End}}
\newcommand{\Ker}{\operatorname{Ker}}
\renewcommand{\Im}{\operatorname{Im}}

\newcommand{\id}{\mathsf{id}}

\newcommand{\xto}{\xrightarrow}
\newcommand{\Sn}{\mathscr{S}_\N}

\def\cone{\operatorname{\mathsf{Cone}}}
\def\cocone{\operatorname{\mathsf{CoCone}}}

 \newcommand{\lra}{\longrightarrow}    
 \newcommand{\dra}{\dashrightarrow}    
\begin{document}
\setlength{\baselineskip}{15pt}
\title[Localization of triangulated categories]{Localization of triangulated categories with respect to extension-closed subcategories}
\author{Yasuaki Ogawa}
\email{ogawa.yasuaki.gh@cc.nara-edu.ac.jp} %
\address{Center for Educational Research of Science and Mathematics, Nara University of Education, Takabatake-cho, Nara, 630-8528, Japan}
\begin{abstract}
The aim of this paper is to develop a framework for localization theory of triangulated categories $\C$, that is, from a given extension-closed subcategory $\N$ of $\C$, we construct a natural extriangulated structure on $\C$ together with an exact functor $Q:\C\to\widetilde{\C}_\N$ satisfying a suitable universality, which unifies several phenomena.
Precisely, a given subcategory $\N$ is thick if and only if the localization $\widetilde{\C}_\N$ corresponds to a triangulated category. In this case, $Q$ is nothing other than the usual Verdier quotient.
Furthermore, it is revealed that $\widetilde{\C}_\N$ is an exact category if and only if $\N$ satisfies a generating condition $\cone(\N,\N)=\C$.
Such an (abelian) exact localization $\widetilde{\C}_\N$ provides a good understanding of some cohomological functors $\C\to\Ab$, e.g., the heart of $t$-structures on $\C$ and the abelian quotient of $\C$ by a cluster-tilting subcategory $\N$.
\end{abstract}
\maketitle

\tableofcontents

\section*{Introduction}\label{sec_intro}
The triangulated category was introduced by Grothendieck and Verdier to formulate the derived category through a more conceptual construction, what we call, the Verdier quotient \cite{Ver96}, which has had an impact not only on the representation theory but also the algebraic geometry and many other mathematical fields.
The Verdier quotient is a triangulated analogue of the Serre quotient of abelian categories \cite{Gab62}.
Both of them are formulated as the Gabriel-Zisman localization with respect to the class $\mathscr{S}_\N$ of morphisms determined by a given thick/Serre subcategory $\N$ such that $\mathscr{S}_\N$ forms a multiplicative system, and thus the localization $\C[\Sn^{-1}]$ becomes easy to handle \cite{GZ67}.
Recently, the notion of extriangulated category was introduced by Nakaoka and Palu \cite{NP19} as a unification of triangulated and exact categories.
The localization theory of extriangulated categories was developed in the pursuit of unification of Verdier/Serre quotients \cite{NOS21},  which tells us that an exact functor $(Q,\mu)\colon (\C,\mathbb{E},\mathfrak{s})\to (\widetilde{\C},\widetilde{\mathbb{E}},\widetilde{\mathfrak{s}})$ of extriangulated categories with suitable universality can be established from an extriangulated category $(\C,\mathbb{E},\mathfrak{s})$ and a certain thick subcategory $\N\subseteq \C$.
The aim of this article is, as an application of such extriangulated localizations, to construct a new framework for localization theory of triangulated categories $\C$. That is, from a given extension-closed subcategory $\N\subseteq\C$, we construct an ``exact'' functor $Q\colon \C\to\widetilde{\C}_\N$ enjoying expected properties.
More precisely, our main results are summarized as follows.

\begin{mainthm}[Theorems~\ref{thm_main1}, \ref{cor_tri}, \ref{cor_exact}]\label{mainthm}
Let $\C$ be a triangulated category and regard it as a natural extriangulated category $(\C,\mathbb{E},\mathfrak{s})$.
Assume that a full subcategory $\N$ of $\C$ is closed under taking extensions, direct summands and isomorphisms.
\begin{enumerate}
\item[\textnormal{(0)}]
The subcategory $\N$ naturally defines a relative extriangulated structure $(\C,\mathbb{E}_\N,\mathfrak{s}_\N)$.
\item[\textnormal{(1)}]
The subcategory $\N$ is thick with respect to the relative structure $(\C,\mathbb{E}_\N,\mathfrak{s}_\N)$.
Moreover, we have an extriangulated localization $(Q,\mu)\colon (\C,\mathbb{E}_\N,\mathfrak{s}_\N)\to (\widetilde{\C}_\N,\widetilde{\mathbb{E}}_\N,\widetilde{\mathfrak{s}}_\N)$.
\item[\textnormal{(2)}]
The subcategory $\N$ is thick in the triangulated category $(\C,\mathbb{E},\mathfrak{s})$ if and only if it is biresolving with respect to the relative structure $(\C,\mathbb{E}_\N,\mathfrak{s}_\N)$ if and only if the resulting category $(\widetilde{\C}_\N,\widetilde{\mathbb{E}}_\N,\widetilde{\mathfrak{s}}_\N)$ is triangulated.
In this case, the localization $(Q,\mu)$ is nothing but the Verdier quotient.
\item[\textnormal{(3)}]
Suppose $\N$ to be functorially finite.
Then, $\N$ satisfies $\cone(\N,\N)=\C$ in the triangulated category $(\C,\mathbb{E},\mathfrak{s})$ if and only if $\N$ is Serre with respect to the relative structure $(\C,\mathbb{E}_\N,\mathfrak{s}_\N)$ if and only if the resulting category $\widetilde{\C}_\N$ is abelian.
Furthermore, the functor $Q\colon (\C,\mathbb{E},\mathfrak{s})\to\widetilde{\C}_\N$ from the original triangulated category is cohomological.
\end{enumerate}
\begin{table}[h]
\centering
  \begin{tabular}{|c|c:c:c|l|}  \hline
    \multirow{2}{*}{$\N$} & extension-closed & thick  & $\cone(\N,\N)=\C$ &in $(\C,\mathbb{E},\mathfrak{s})$\\ \cline{2-5} 
      & thick & biresolving & Serre &in $(\C,\mathbb{E}_\N,\mathfrak{s}_\N)$\\ \hline
     $\cellcolor[gray]{0.9}\widetilde{\C}_\N$ & \cellcolor[gray]{0.9}extriangulated & \cellcolor[gray]{0.9}\ \ triangulated\ \ & \cellcolor[gray]{0.9} abelian &\diagbox[width=2.5cm]{\ }{\ }\\ \cline{1-5}
  \end{tabular}
\end{table}
\end{mainthm}
\noindent
Our construction simultaneously contains the Verdier quotient and certain types of cohomological functors as mentioned below.

As well as the Verdier/Serre quotient, other localizations with respect to a multiplicative system appear in various contexts.
Rump showed that, as for an integral preabelian category $\C$, the class $\mathscr{S}$ of monic and epic morphisms forms a multiplicative system and the localization $\C[\mathscr{S}^{-1}]$ is abelian \cite{Rum01}.
In view of Rump's localization, Beligiannis and Buan-Marsh investigated the natural cohomological functor $\C(T,-)\colon \C\to \mod\End_\C(T)^{\op}$ for a rigid object $T$ in a triangulated category $\C$.
They proved that the extension-closed subcategory $\N:=\Ker\C(T,-)$ of $\C$ naturally determines a class $\mathscr{S}_\N$ of morphisms and there exists a natural equivalence $\C[\Sn^{-1}]\simeq\mod\End_\C(T)^{\op}$.
They also proved that, although $\mathscr{S}_\N$ is not a multiplicative system in $\C$, but $\overline{\Sn}$ becomes a multiplicative system in $\overline{\C}:=\C/[\N]$ the ideal quotient by the morphisms factoring through objects in $\N$ \cite{Bel13, BM13a, BM13b}.
Many authors push this perspective further in various directions by using Abe-Nakaoka's heart of (twin) cotorsion pairs \cite{Nak13, LN19, Liu17, HS20, Oga22}.
Whereas J{\o}rgensen, Palu and Shah investigated relations between a relative extriangulated structure on $\C$ determined by $\N$ and an abelian exact structure on $\mod\End_\C(T)^{\op}$ in terms of the Grothendieck group \cite{JP22, JS22} (see also \cite{Pal08, PPPP19}).
Our localization recovers and strengthens some relevant results in the sense that
it explains how the abelian exact structure on $\mod\End_\C(T)^{\op}$ inherits from the triangulated category $\C$.
In fact, it also streamlines the constructions of a wider class of cohomological functors containing the above mentioned ones.

This article is organized as follows.
The first section is devoted to recall needed foundations of the localization theory of extriangulated categories.
In Section 2, we formulate our extriangulated localization $\widetilde{\C}_\N$ of a triangulated category $\C$ with respect to an extension-closed subcategory $\N\subseteq \C$.
In Section 3 and 4, we reveal necessary and sufficient conditions for the resulting category $\widetilde{\C}_\N$ to be triangulated and exact.
In Section 5, we demonstrate relevance of our localization for 
some cohomological functors.

\subsection*{Notation and convention}
All categories and functors in this article are always assumed to be additive (except for Subsection 1.1).
For an additive category $\C$, the set of morphisms $A\rightarrow B$ in $\C$ is denoted by $\C(A,B)$ or more simply denoted by $(A,B)$ if there is no confusion.
The classes of objects and morphisms in $\C$ are denoted by $\Ob\C$ and $\Mor\C$, respectively.
The symbol $\C^{\op}$ stands for the opposite category of $\C$.
All subcategory $\U\subseteq \C$ is always assumed to be full, additive and closed under isomorphisms.
For $X\in\C$, if $\C(U,X)=0$ for any $U\in\U$, we write abbreviately $\C(\U,X)=0$.
Similar notations will be used in obvious meanings.

\section{Preliminary}\label{sec_pre}

\subsection{Gabriel-Zisman localization}\label{ssec:Gabriel-Zisman_localization}
Since we are interested in Gabriel-Zisman localizations of extriangulated categories which are still equipped with extriangulated structures, in this subsection, we recall the definition of the Gabriel-Zisman localization, following \cite{Fri08} (see also \cite[Section I.2]{GZ67}).

\begin{definition}
Let $\C$ and $\D$ be categories and $\mathscr{S}$ a class of morphisms in $\C$.
A functor $Q\colon \C\to\D$ is called a \textit{Gabriel-Zisman localization of $\C$ with respect to $\mathscr{S}$} if the following universality holds:
\begin{enumerate}
\item[\textnormal{(1)}] $Q(s)$ is an isomorphism in $\D$ for any $s\in\mathscr{S}$;
\item[\textnormal{(2)}] For any functor $F\colon \C\to\D'$ which sends each morphism in $\mathscr{S}$ to an isomorphism in $\D'$, there uniquely exists a functor $F'\colon \D\to\D'$ such that $F=F'\circ Q$.
\end{enumerate}
In this case, we denote the localization functor by $Q\colon \C\to\D=\C[\mathscr{S}^{-1}]$.
\end{definition}

Since the Gabriel-Zisman localization $\C[\mathscr{S}^{-1}]$ always exists provided there are no set-theoretic obstructions,
whenever we consider the Gabriel-Zisman localization of $\C$, we assume that $\C$ is skeletally small.
Morphisms in the category $\C[\mathscr{S}^{-1}]$ can be regarded as compositions of the original morphisms and the formal inverses, see \cite[Thm. 2.1]{Fri08} for the details.
However such a description is hard to control.
If the class $\mathscr{S}$ satisfies the following conditions and its dual, 
any morphism in $\C[\mathscr{S}^{-1}]$ has a very nice description.

\begin{definition}
Let $\mathscr{S}$ be a class of morphisms in an additive category $\C$.
\begin{enumerate}
\item[(MS0)] The identity morphisms of $\C$ lie in $\mathscr{S}$ and $\mathscr{S}$ is closed under composition.
Also, $\mathscr{S}$ is closed under taking finite direct sums. Namely, if $f_i\in\mathscr{S}(X_i,Y_i)$ for $i=1,2$, then $f_1\oplus f_2\in\mathscr{S}(X_1\oplus X_1,Y_1\oplus Y_2)$.
\item[(MS1)] Any diagram of the form
\[
\xymatrix@R=16pt@C=16pt{
&B\ar[d]^f\\
C\ar[r]^s&D
}
\]
with $s\in\mathscr{S}$ can be completed in a commutative square
\[
\xymatrix@R=16pt@C=16pt{
A\ar[r]^{s'}\ar[d]_{f'}&B\ar[d]^f\\
C\ar[r]^s&D
}
\]
with $s'\in\mathscr{S}$.
\item[(MS2)] If $s\colon Y\to Y'$ in $\mathscr{S}$ and $f,f':X\to Y$ are morphisms such that $sf=sf'$, then there exists $s'\colon X'\to X$ in $\mathscr{S}$ such that $fs'=f's'$.
\end{enumerate}
If the above conditions and the dual are satisfied, $\mathscr{S}$ is called a \emph{multiplicative system} of $\C$.
We also say that the class $\mathscr{S}$ \emph{admits a calculus of left and right fractions}.
In this case, if $\C$ is additive, then so is the resulting category $\C[\mathscr{S}^{-1}]$. Moreover, the quotient functor $Q\colon \C\to \C[\S^{-1}]$ is an additive functor.
\end{definition}

\begin{remark}
The condition {\rm (MS0)} is a bit different from the original one defined in \cite{GZ67}.
However, since all categories considered in this article are additive, the difference does not affect on the formulation of the category $\C[\mathscr{S}^{-1}]$.
\end{remark}

As is well-known, the Verdier localization of a triangulated category $\C$ with respect to its thick subcategory $\N$ is defined to be a Gabriel-Zisman localization which admits a calculus of left and right fractions.
In fact, it is the localization of $\C$ with respect to the multiplicative system $\Sn$ consisting of morphisms whose cones belong to $\N$  (e.g. \cite{Nee01}).
The aim of this article is to expand the Verdier localization to a wider case that a given subcategory $\N$ is not thick but still extension-closed.

\subsection{Extriangulated categories}\label{ssec_ET}
This subsection is devoted to recall the localization theory of extriangulated category with respect to a suitable class of morphisms which was introduced in \cite{NOS21}, in the pursuit of unifying the Verdier and Serre quotients.
First we briefly recall some terminology and basic results on extriangulated categories.
All extriangulated categories we consider in this article arise from a given triangulated category $\C$ as a relative theory \cite{INP19} or a localization \cite{NOS21} of $\C$. So, in fact, we do not use the axioms of extriangulated categories and we do not include all the details (see \cite[Sect. 2]{NP19}).
An \textit{extriangulated category} is defined to be a triple $(\C,\mathbb{E},\mathfrak{s})$ of
\begin{itemize}
\item[-] an additive category $\C$;
\item[-] a biadditive bifunctor $\mathbb{E}\colon \C^{\op}\times\C\to \Ab$, where $\Ab$ is the category of abelian groups;
\item[-] a correspondence $\mathfrak{s}$ which associates each equivalence class of  a sequence  $A\lra B\lra C$ in $\C$ to an element in $\mathbb{E}(C,A)$ for any $C,A\in\C$,
\end{itemize}
which satisfies some `additivity' and `compatibility'.
An equivalence class of $A\lra B\lra C$ is denoted by $[A\lra B\lra C]$.
An extriangulated category $(\C,\mathbb{E},\mathfrak{s})$ is simply denoted by $\C$ if there is no confusion.
We also recall the following basic terminology which will be used in many places.

\begin{definition}
Let $(\C,\mathbb{E},\mathfrak{s})$ be an extriangulated category.
\begin{enumerate}
\item[\textnormal{(1)}] An element $\delta\in\mathbb{E}(C, A)$ is called an \emph{$\mathbb{E}$-extension} for any $A,C\in\C$.
For $a\in\C(A,A')$ and $c\in\C(C',C)$, we write as $a_*\delta =\mathbb{E}(C,a)(\delta)$ and $c^*\delta =\mathbb{E}(c,A)(\delta)$.
\item[\textnormal{(2)}] A sequence $A\overset{f}{\lra}B\overset{g}{\lra}C$ corresponding to an $\mathbb{E}$-extension $\delta\in\mathbb{E}(C,A)$ is called an $\mathfrak{s}$-\textit{conflation}.
 In addition, $f$ and $g$ are called an \emph{$\mathfrak{s}$-inflation} and an \emph{$\mathfrak{s}$-deflation}, respectively.
A pair $\langle A\overset{f}{\lra}B\overset{g}{\lra}C, \delta\rangle$ of $\mathbb{E}$-extension $\delta$ and a corresponding $\mathfrak{s}$-conflation is also denoted by $A\overset{f}{\lra}B\overset{g}{\lra}C\overset{\delta}{\dashrightarrow}$ and we call it an \emph{$\mathfrak{s}$-triangle}.
\item[\textnormal{(3)}]
A \emph{morphism} of $\mathfrak{s}$-triangles from $\langle A\overset{f}{\lra}B\overset{g}{\lra}C, \delta\rangle$ to $\langle A'\overset{f'}{\lra}B'\overset{g'}{\lra}C', \delta' \rangle$ is a triplet $(a,b,c)$ of morphisms which makes the following diagram commutative and satisfies $a_*\delta=c^*\delta'$.
\[
 \xymatrix{
 A\ar[r]^{f}\ar[d]_a&B\ar[r]^{g}\ar[d]_b&C\ar[d]^c\ar@{-->}[r]^{\delta}&\\
 A'\ar[r]^{f'}&B'\ar[r]^{g'}&C'\ar@{-->}[r]^{\delta'}&
 }
\]
\end{enumerate}
\end{definition}

Keeping in mind the triangulated case, we introduce the following notions.

\begin{proposition}
Let $\C$ be an extriangulated category.
For an $\mathfrak{s}$-inflation $f\in\C(A,B)$, there exists an $\mathfrak{s}$-triangle $A\overset{f}{\lra} B\overset{g}{\lra} C\overset{\delta}{\dra}$.
Then we call $C$ a \emph{cone} of $f$ and put $\cone(f):=C$.
Similarly, we denote the object $A$ by $\cocone(g)$ and call it a \emph{cocone} of $g$.
For a given $\mathfrak{s}$-inflation $f$, $\cone(f)$ is uniquely determined up to isomorphism.
A dual statement holds for $\cocone(g)$.
\end{proposition}

For any subcategories $\U$ and $\V$ in $\C$, we denote by $\cone(\V,\U)$ the subcategory consisting of objects $X$ appearing in an $\mathfrak{s}$-triangle $V\lra U\lra X\dra $ with $U\in\U$ and $V\in\V$.
A subcategory $\cocone(\V,\U)$ is defined similarly.

The next lemma directly follows from \cite[Prop.~1.20]{LN19} (see also \cite[Cor.~3.16]{NP19}, \cite[Prop.~2.22]{HS20}) which shows that, like triangulated categories, extriangulated categories admit weak pushouts and weak pullbacks.
It will be used in many places.

\begin{lemma}
\label{lem:wPO_wPB}
The following properties hold.
\begin{enumerate}
\item[\textnormal{(1)}]
For any $\mathfrak{s}$-triangle $A\overset{f}\lra B\overset{g}{\lra} C\overset{\delta}{\dra} $ together with a morphism $a\in\C(A, A')$ and the associated $\mathfrak{s}$-triangle $A'\overset{f'}\lra B'\overset{g'}{\lra} C\overset{a_*\delta}{\dra} $,
there exists $b\in\C(B,B')$ which gives a morphism of $\mathfrak{s}$-triangles 
\[
 \xymatrix{
 A\ar@{}[rd]|{\textnormal{(wPO)}}\ar[r]^{f}\ar[d]_a&B\ar[r]^{g}\ar[d]^b&C\ar@{=}[d]\ar@{-->}[r]^{\delta}&\\
 A'\ar[r]^{f'}&B'\ar[r]^{g'}&C\ar@{-->}[r]^{a_*\delta}&
 }
\]
and makes $A\overset{\binom{f}{a}}{\lra} B\oplus A'\overset{(b\ -f')}{\lra} B'\overset{g'^*\delta}{\dra}$ an $\mathfrak{s}$-triangle.
Furthermore, the commutative square ${\rm (wPO)}$ is a \emph{weak pushout of $f$ along $a$} in the following sense. 
Given morphisms $d\in\C(A',D)$ and $e\in\C(B,D)$ with $da=ef$, there exists $h\in\C(B',D)$ such that $hb=d$ and $hf'=e$. 
\item[\textnormal{(2)}]
Dually, for any $\mathfrak{s}$-triangle $A\overset{f}\lra B\overset{g}{\lra} C\overset{\delta}{\dra} $ together with a morphism $c\in\C(C',C)$ and the associated $\mathfrak{s}$-triangle $A\overset{f'}\lra B'\overset{g'}{\lra} C'\overset{c^*\delta}{\dra} $,
there exists $b\in\C(B',B)$ which gives a morphism of $\mathfrak{s}$-triangles
\[
 \xymatrix{
 A\ar[r]^{f'}\ar@{=}[d]&B'\ar@{}[rd]|{\textnormal{(wPB)}}\ar[r]^{g'}\ar[d]_b&C'\ar[d]^c\ar@{-->}[r]^{c^*\delta}&\\
 A\ar[r]^{f}&B\ar[r]^{g}&C\ar@{-->}[r]^{\delta}&
 }
\]
and makes $B'\overset{\binom{-g'}{b}}{\lra} C'\oplus B\overset{(c\ g)}{\lra} C\overset{f'_*\delta}{\dra}$ an $\mathfrak{s}$-triangle.
The commutative square ${\rm (wPB)}$ is a \emph{weak pullback of $g$ along $c$}, which is defined dually to a weak pushout.
\end{enumerate}
\end{lemma}


Triangulated/exact structures on a additive category $\C$ naturally give rise to extriangulated structures.
In this case, we say that the extriangulated category \emph{corresponds to a triangulated/exact category} (see \cite[Prop.~3.22, Exam.~2.13]{NP19} for details).
Moreover, there exist useful criteria for an extriangulated category to correspond to a triangulated/exact category.

\begin{proposition}\label{prop_extri_to_tri}
Let $(\C, \mathbb{E}, \mathfrak{s})$ be an extriangulated category, in which any morphisms are $\mathfrak{s}$-inflations and $\mathfrak{s}$-deflations.
Then, it corresponds to a triangulated category.
In this case, for any triangle $A\overset{f}{\lra}B\overset{g}{\lra}C\overset{h}{\lra}A[1]$,
we denote by $A\overset{f}{\lra}B\overset{g}{\lra}C\overset{h}{\dashrightarrow}$ the corresponding $\mathfrak{s}$-triangle.
\end{proposition}
\begin{proof}
It directly follows from \cite[Thm. 6.20]{NP19} and \cite[Cor. 7.6]{NP19}.
\end{proof}

\begin{proposition}\cite[Cor. 3.18]{NP19}\label{prop_extri_to_exact}
Let $(\C, \mathbb{E}, \mathfrak{s})$ be an extriangulated category, in which any $\mathfrak{s}$-inflation is monic and any $\mathfrak{s}$-deflation is epic.
Then, it corresponds to an exact category.
\end{proposition}

One of the advantages of revealing an extriangulated structure lies in the fact that it is closed under certain basic operations: extension-closed subcategories; relative theory; localization.

\subsubsection{Extension-closed subcategories}
The first one says that an extension-closed subcategory of an extriangulated category is still equipped with an extriangulated structure.

\begin{proposition}\cite[Rem. 2.18]{NP19}\label{prop:closed_under_operation}
Let $(\C,\mathbb{E},\mathfrak{s})$ be an extriangulated category and $\N$ an extension-closed subcategory of $\C$.
If we define $\mathbb{E}|_\N$ to be the restriction of $\mathbb{E}$ to $\N^{\op}\times \N$ and $\mathfrak{s}|_\N:=\mathfrak{s}|_{\mathbb{E}|_\N}$, then $(\N,\mathbb{E}|_\N,\mathfrak{s}|_\N)$ becomes an extriangulated category.
\end{proposition}

\subsubsection{Relative theory}
The second one is a relative structure of an extriangulated category.
A ``weaker'' structure with respect to a given triangulated/exact category is considered in various contexts, e.g., \cite{Bel00, Kra00} for triangulated cases and \cite{AS93, DRSSK99} for exact cases.
These concepts are generalized in terms of extriangulated categories as below.

\begin{proposition}\label{prop_relative}\cite[Prop. 3.16]{HLN21}
Let us consider an extriangulated category $(\C,\mathbb{E},\mathfrak{s})$.
The following conditions are equivalent for a biadditive subfunctor $\mathbb{F}\subseteq \mathbb{E}$.
\begin{enumerate}
\item[{\rm (1)}]
$(\C,\mathbb{F},\mathfrak{s}|_\mathbb{F})$ forms an extriangulated category, where $\mathfrak{s}|_\mathbb{F}$ is the restriction of $\mathfrak{s}$ to $\mathbb{F}$.
\item[{\rm (2)}]
$\mathfrak{s}|_\mathbb{F}$-inflations are closed under compositions.
\item[{\rm (3)}]
$\mathfrak{s}|_\mathbb{F}$-deflations are closed under compositions.
\end{enumerate}
If $\mathbb{F}$ satisfies the above equivalent conditions, it is called a \emph{closed} subfunctor of $\mathbb{E}$.
Furthermore, we say that the extriangulated structure $(\C,\mathbb{F},\mathfrak{s}|_\mathbb{F})$ is \emph{relative} to $(\C,\mathbb{E},\mathfrak{s})$ or a \emph{relative theory} of $(\C,\mathbb{E},\mathfrak{s})$.
\end{proposition}

\subsubsection{Localization}
The third operation is a certain localization which is instrumental to formulate our statements.
In advance, we recall from \cite[Def. 2.11]{NOS21} the definition of exact functors between extriangulated categories, see also \cite[Lem. 3.19]{BTHSS23}.

\begin{definition}\label{Def_exact_functor}
Let $(\C,\mathbb{E},\mathfrak{s})$, $(\C',\mathbb{E}',\mathfrak{s}')$ and $(\C'',\mathbb{E}'',\mathfrak{s}'')$ be extriangulated categories.
\begin{enumerate}
\item
{\rm (\cite[Def. 2.23]{B-TS21})}
An \emph{exact functor} $(F,\phi)\colon (\C,\mathbb{E},\mathfrak{s})\to(\C',\mathbb{E}',\mathfrak{s}')$ is a pair of an additive functor $F\colon \C\to\C'$ and a natural transformation $\phi\colon \mathbb{E}\rightarrow\mathbb{E}'\circ(F^{\op}\times F)$ which satisfies
\[
\mathfrak{s}'(\phi_{C,A}(\delta))=[F(A)\overset{F(x)}{\lra}F(B)\overset{F(y)}{\lra}F(C)]
\]
for any $\mathfrak{s}$-triangle $A\overset{x}{\lra}B\overset{y}{\lra}C\overset{\delta}{\dashrightarrow}$ in $\C$.
\item
For exact functors $(F,\phi)\colon (\C,\mathbb{E},\mathfrak{s})\to (\C',\mathbb{E}',\mathfrak{s}')$ and $(F',\phi')\colon (\C',\mathbb{E}',\mathfrak{s}')\to (\C'',\mathbb{E}'',\mathfrak{s}'')$, their composition $(F'',\phi'')=(F'\phi')\circ(F,\phi)$ is defined to be the pair of $F''= F'\circ F$ and $\phi''= (\phi'\cdot(F^{\op}\times F))\circ\phi$, where $\phi'\cdot(F^{\op}\times F)$ denotes the whiskering of $\phi'$ with $F^{\op}\times F$.
\item
Let $(F,\phi),(G,\psi)\colon (\C,\mathbb{E},\mathfrak{s})\to (\C',\mathbb{E}',\mathfrak{s}')$ be exact functors. A \emph{natural transformation} $\eta\colon (F,\phi)\to (G,\psi)$ \emph{of exact functors} is a natural transformation $\eta\colon F\to G$ of additive functors, which satisfies
\begin{equation*}
(\eta_A)_*\phi_{C,A}(\delta)=(\eta_C)^*\psi_{C,A}(\delta)
\end{equation*}
for any $\delta\in\mathbb{E}(C,A)$.
\item
{\rm (\cite[Prop. 2.13]{NOS21})}
An exact functor $(F,\phi)\colon (\C,\mathbb{E},\mathfrak{s})\to(\C',\mathbb{E}',\mathfrak{s}')$ is called an \emph{exact equivalence}, if there exists an exact functor $(G,\psi)\colon (\C',\mathbb{E}',\mathfrak{s}')\to(\C,\mathbb{E},\mathfrak{s})$ and a natural isomorphisms $\eta\colon (G,\psi)\circ (F,\phi)\cong\id_{\C}$ and $\eta'\colon (F,\phi)\circ(G,\psi)\cong\id_{\C'}$ of exact functors.
\end{enumerate}
\end{definition}

\begin{remark}\label{Rem_exact_functor}
\begin{enumerate}
\item
The notion of an exact functor coincides with the usual ones when both of $(\C,\mathbb{E},\mathfrak{s})$ and $(\C',\mathbb{E}',\mathfrak{s}')$ correspond to exact or triangulated categories, see \cite[Thm.~2.33, 2.34]{B-TS21} for the details.
\item
A triangle equivalence of triangulated categories gives rise to an exact equivalence with respect to the corresponding extriangulated structures.
\item
Similarly, an equivalence of abelian categories can be regarded as an exact equivalence in an obvious way.
\end{enumerate}
\end{remark}

Extension-closedness and the relative theory provide typical examples of exact functors, which is mentioned in \cite[Exam. 3.8]{Hau21} from the viewpoint of $n$-extriangulated subcategories.

\begin{example}
\label{ex_exact_functor}
Let $(\C,\mathbb{E},\mathfrak{s})$ be any extriangulated category. 
\begin{enumerate}
\item
We regard  an extension-closed subcategory $\N\subseteq\C$ as a natural extriangulated category $(\N,\mathbb{E}|_\N,\mathfrak{s}|_\N)$.
The canonical inclusion functor $\mathsf{inc}\colon\N\hookrightarrow\C$ induces an exact functor 
$(\mathsf{inc},\iota)\colon (\N,\mathbb{E}|_\N,\mathfrak{s}|_\N)\to(\C,\mathbb{E},\mathfrak{s})$, where $\iota\colon \mathbb{E}|_\N \rightarrow \mathbb{E}$ is the canonical inclusion. 

\item
For a closed subfunctor $\mathbb{F}\subseteq\mathbb{E}$ and the relative extriangulated category $(\C,\mathbb{F},\mathfrak{s}|_\mathbb{F})$, 
the identity $\id_{\C}$ and the inclusion $\iota \colon\mathbb{F}\hookrightarrow\mathbb{E}$ give rise to an exact functor $(\id_{\C},\iota)\colon (\C,\mathbb{F},\mathfrak{s}|_\mathbb{F})\to(\C,\mathbb{E},\mathfrak{s})$. 

\end{enumerate}
\end{example}

In the rest of this section, we recall the main result of \cite{NOS21} which provides a sufficient condition for an extriangulated category $\C$ and a class $\mathscr{S}$ of morphisms to impose a natural extriangulated structure on the Gabriel-Zisman localization $\widetilde{\C}:=\C[\mathscr{S}^{-1}]$.
We fix a class $\mathscr{S}$ of morphisms in $\C$ which satisfies {\rm (MS0)} and associate a full subcategory $\N_{\mathscr{S}}\subseteq\C$ as below.

\begin{definition}\label{Def_NS}
Let $\mathscr{S}$ be the above. Define $\N_{\mathscr{S}}\subseteq\C$ to be the full subcategory consisting of objects $N\in\C$ such that both $N\to 0$ and $0\to N$ belong to $\mathscr{S}$. 
It is obvious that $\N_{\mathscr{S}}\subseteq\C$ is an additive subcategory.
\end{definition}

For any subcategory $\N\subseteq\C$,
we will denote the ideal quotient by $p\colon \C\to \overline{\C}=\C/[\N]$, and $\overline{f}$ will denote a morphism in $\overline{\C}$ represented by $f\in\C(X,Y)$.
We put $\overline{\mathscr{S}}:=\{\overline{s}\mid s\in\mathscr{S}\}$.
Note that $\overline{f}=0$ if and only if $f$ factors through an object in $\N$.
In this case, we write $f\in[\N]$.
Also, let $\overline{\mathscr{S}}^*$ be the closure of $\overline{\mathscr{S}}$ with respect to compositions with isomorphisms in $\overline{\C}$.

\begin{theorem}\cite[Thm. 3.5]{NOS21}\label{Thm_Mult_Loc}
Let $\mathscr{S}$ be a class of morphisms with {\rm (MS0)} and consider the ideal quotient $p\colon \C\to\overline{\C}:=\C/[\N_\mathscr{S}]$ and the localization $Q\colon \C\to\C[\mathscr{S}^{-1}]=:\widetilde{\C}$.
We assume that $\overline{\mathscr{S}}=\overline{\mathscr{S}}^*$ holds.
\begin{enumerate}
\item[{\rm (1)}]
Suppose that $\overline{\mathscr{S}}$ satisfies the following conditions {\rm (MR1),\ldots,(MR4)}.
Then we obtain an extriangulated category $(\widetilde{\C},\widetilde{\mathbb{E}},\widetilde{\mathfrak{s}})$ together with an exact functor $(Q,\mu)\colon (\C,\mathbb{E},\mathfrak{s})\to (\widetilde{\C},\widetilde{\mathbb{E}},\widetilde{\mathfrak{s}})$.
\begin{itemize}
\item[{\rm (MR1)}] $\overline{\mathscr{S}}$ satisfies $2$-out-of-$3$ with respect to compositions in $\overline{\C}$, i.e., for any composed morphism $g\circ f$, if two of $\{f,g,g\circ f\}$ belong to $\overline{\mathscr{S}}$, then so does the third.
\item[{\rm (MR2)}] $\overline{\mathscr{S}}$ is a multiplicative system in $\overline{\C}$.
\item[{\rm (MR3)}] Let $\langle A\overset{x}{\lra}B\overset{y}{\lra}C,\delta\rangle$, $\langle A'\overset{x'}{\lra}B'\overset{y'}{\lra}C',\delta'\rangle$ be any pair of $\mathfrak{s}$-triangles, and let $a\in\C(A,A'),c\in\C(C,C')$ be any pair of morphisms satisfying $a_*\delta=c^*\delta'$. If $\overline{a}$ and $\overline{c}$ belong to $\overline{\mathscr{S}}$, then there exists $\overline{b}\in\overline{\mathscr{S}}(B,B')$ which satisfies $\overline{b}\circ\overline{x}=\overline{x}'\circ\overline{a}$ and $\overline{c}\circ\overline{y}=\overline{y}'\circ\overline{b}$.
\item[{\rm (MR4)}] $\overline{\mathcal{M}}_{\mathsf{inf}}:=\{ \overline{v}\circ \overline{x}\circ \overline{u}\mid x\ \text{is an}\ \mathfrak{s}\text{-inflation}, \overline{u},\overline{v}\in\overline{\mathscr{S}}\}$ is closed under composition in $\overline{\C}$.
Dually, $\overline{\mathcal{M}}_{\mathsf{def}}:=\{ \overline{v}\circ \overline{y}\circ \overline{u}\mid y\ \text{is an}\ \mathfrak{s}\text{-deflation}, \overline{u},\overline{v}\in\overline{\mathscr{S}}\}$ is closed under compositions.
\end{itemize}

\item[{\rm (2)}]
The exact functor $(Q,\mu)\colon (\C,\mathbb{E},\mathfrak{s})\to (\widetilde{\C},\widetilde{\mathbb{E}},\widetilde{\mathfrak{s}})$ obtained in {\rm (1)} is characterized by the following universality.
\begin{itemize}
\item[{\rm (i)}]
For any exact functor $(F,\phi)\colon (\C,\mathbb{E},\mathfrak{s})\to (\D,\mathbb{F},\mathfrak{t})$ such that $F(s)$ is an isomorphism for any $s\in\mathscr{S}$, there exists a unique exact functor $(\widetilde{F},\widetilde{\phi})\colon (\widetilde{\C},\widetilde{\mathbb{E}},\widetilde{\mathfrak{s}})\to (\D,\mathbb{F},\mathfrak{t})$ with $(F,\phi)=(\widetilde{F},\widetilde{\phi})\circ (Q,\mu)$.
\item[{\rm (ii)}] For any pair of exact functors $(F,\phi),(G,\psi)\colon (\C,\mathbb{E},\mathfrak{s})\to (\D,\mathbb{F},\mathfrak{t})$ which send  any $s\in\mathscr{S}$ to isomorphisms, let $(\widetilde{F},\widetilde{\phi}),(\widetilde{G},\widetilde{\psi})\colon (\widetilde{\C},\widetilde{\mathbb{E}},\widetilde{\mathfrak{s}})\to (\D,\mathbb{F},\mathfrak{t})$ be the exact functors obtained in {\rm (i)}. Then for any natural transformation $\eta\colon (F,\phi)\to(G,\psi)$ of exact functors, there uniquely exists a natural transformation $\widetilde{\eta}\colon (\widetilde{F},\widetilde{\phi})\to(\widetilde{G},\widetilde{\psi})$ of exact functors satisfying $\eta=\widetilde{\eta}\circ (Q,\mu)$.
\end{itemize}
\end{enumerate}
The exact functor $(Q,\mu)\colon (\C,\mathbb{E},\mathfrak{s})\to (\widetilde{\C},\widetilde{\mathbb{E}},\widetilde{\mathfrak{s}})$ or the resulting category $(\widetilde{\C},\widetilde{\mathbb{E}},\widetilde{\mathfrak{s}})$ obtained in {\rm (1)} is called an \emph{extriangulated localization} of $\C$ with respect to $\mathscr{S}$.
\end{theorem}

We refer to \cite[Sect. 3]{NOS21} for explicit forms of $\widetilde{\mathbb{E}}$ and $\widetilde{\mathfrak{s}}$.
By definition, any $\widetilde{\mathfrak{s}}$-inflation (resp. $\widetilde{\mathfrak{s}}$-deflation) comes from an $\mathfrak{s}$-inflation (resp. $\mathfrak{s}$-deflation) in the original category $(\C, \mathbb{E}, \mathfrak{s})$.
More precisely, we have the lemma below.

\begin{lemma}\cite[Lem. 3.32]{NOS21}\label{lem_inf}
The following holds for any morphism $\alpha$ in $(\widetilde{\C},\widetilde{\mathbb{E}},\widetilde{\mathfrak{s}})$.
\begin{enumerate}
\item[{\rm (1)}] $\alpha$ is an $\widetilde{\mathfrak{s}}$-inflation in $\widetilde{\C}$ if and only if $\alpha=\beta\circ Q(f)\circ \gamma$ holds for some $\mathfrak{s}$-inflation $f$ in $\C$ and isomorphisms $\beta,\gamma$ in $\widetilde{\C}$.
\item[{\rm (2)}] $\alpha$ is an $\widetilde{\mathfrak{s}}$-deflation in $\widetilde{\C}$ if and only if $\alpha=\beta\circ Q(f)\circ\gamma$ holds for some $\mathfrak{s}$-deflation $f$ in $\C$ and isomorphisms $\beta,\gamma$ in $\widetilde{\C}$.
\end{enumerate}
\end{lemma}

From a certain subcategory $\N\subseteq \C$,
we can construct a class $\Sn$ of morphisms which satisfies ${\rm (MR1)},\ldots, {\rm (MR4)}$.
Before listing such subcategories $\N$, we introduce the notion of a thick subcategory of an extriangulated category \cite[Def. 4.1]{NOS21}.

\begin{definition}\label{def_thick}
A full subcategory $\N$ of $\C$ is called a \emph{thick} subcategory if it satisfies the following conditions.
\begin{enumerate}
\item
$\N$ is closed under direct summands and isomorphisms.
\item
$\N$ satisfies $2$-out-of-$3$ for $\mathfrak{s}$-conflations. Namely, for any $\mathfrak{s}$-conflation $A\lra B\lra C$, if two of $\{A, B, C\}$ belong to $\N$, then so does the third.
\end{enumerate}
Moreover, a thick subcategory $\N$ is said to be:
\begin{itemize}
\item[-]
\emph{biresolving}, if, for any object $C\in\C$, there exist an $\mathfrak{s}$-inflation $C\to N$ and an $\mathfrak{s}$-deflation $N'\to C$ with $N,N'\in\N$;
\item[-]
\emph{Serre}, if there exists an $\mathfrak{s}$-conflation $A\lra B\lra C$ with $B\in\N$, then we have $A,C\in\N$.
\end{itemize}
\end{definition}

Note that, in the case $\C$ corresponds to a triangulated category, the notion of thick subcategory coincides with the usual one.
Similarly, in the case that $\C$ corresponds to an exact category,
the notion of biresolving (resp. Serre) subcategory coincides with the one defined in \cite[p. 403]{Rum21} (resp. \cite[4.0.35]{C-E98}).

\begin{definition}\label{def_Sn_from_thick}
For a thick subcategory $\N$, we associate the following classes
of morphisms.
\begin{enumerate}
\item
$\L = \{f\in\Mor\C \mid f \ \textnormal{is an $\mathfrak{s}$-inflation with}\ \cone(f)\in\N\}$.
\item
$\R = \{g\in\Mor\C \mid g \ \textnormal{is an $\mathfrak{s}$-deflation with}\ \cocone(g)\in\N\}$.
\end{enumerate}
Define $\Sn$ to be the smallest subclass closed by compositions containing both $\L$ and $\R$.
\end{definition}

It is mentioned in \cite{NOS21} without proof that $\Sn$ satisfies condition {\rm (MS0)}.
The closedness under compositions is obvious, but to be precise, we include a short argument on the closedness under direct sums.
\begin{lemma}\label{lem:Sn_is_closed_under_direct_sums}
The class $\Sn$ is closed under taking finite direct sums.
\end{lemma}
\begin{proof}
Since $\N$ is additive, we can easily check that both $\L$ and $\R$ are closed under taking finite direct sums.
Let us consider a morphism $A_1\xto{f_1}B_1$ in $\L$ and $A_2\xto{f_2}B_2$ in $\R$.
We will show that $f_1\oplus f_2$ belongs to $\Sn$.
Note that there exists a factorization of $f_1\oplus f_2$ as below.
\[
A_1\oplus A_2\xto{f_1\oplus \id_{A_2}}B_1\oplus A_2\xto{\id_{B_1}\oplus f_2}B_1\oplus B_2.
\]
Since the identity belongs to $\L\cap \R$, we know $f\oplus \id_{A_2}\in\L$ and $\id_{B_1}\oplus f_2\in\R$ as mentioned above.
Thus, as $\Sn$ is closed under compositions, $f_1\oplus f_2$ still belongs to $\Sn$.
By mimicking the argument, we can prove the assertion inductively.
\end{proof}

The following is a list of subcategories $\N$ which induces an extriangulated localization with respect to $\Sn$, see \cite[Exam.~4.8, 4.9, 4.17, 4.33]{NOS21}.

\begin{example}\label{ex_list}
\begin{enumerate}
\item (Verdier quotient.) Let $\C$ be a triangulated category and $\N$ a thick subcategory of $\C$.
Then, the associated class $\Sn$ of morphisms satisfies ${\rm (MS0)} ,{\rm (MR1),\cdots,{\rm (MR4)}}$ and $\Sn=\L=\R$.
The extriangulated localization $Q$ is nothing but the usual Verdier quotient.

\item (Serre quotient.) Let $\C$ be an abelian category and $\N$ a Serre subcategory of $\C$.
Then, we have $\Sn=\L\circ\R$ which admits an extriangulated localization $Q$.
Such a localization is nothing but the usual Serre quotient.

\item (Biresolving quotient.) Let $\C$ be an extriangulated category and $\N$ a biresolving subcategory of $\C$.
Then, we have $\Sn=\L=\R$ which admits an extriangulated localization $Q$.
In this case, the resulting category $(\widetilde{\C},\widetilde{\mathbb{E}},\widetilde{\mathfrak{s}})$ corresponds to a triangulated category \cite[Cor.~4.27]{NOS21}, which is an extriangulated version of \cite[Thm. 5]{Rum21}.
\end{enumerate}
\end{example}

\section{Localization with respect to extension-closed subcategories}\label{sec_Localization}
\emph{In the rest, let $\C$ be a triangulated category with suspension $[1]$ and $\N$ a full subcategory of $\C$ which is closed under direct summands, isomorphisms and extensions.}
This section is devoted to formulate a localization of $\C$ with respect to $\N$ such that the resulting category $\widetilde{\C}_\N$ inherits a natural extriangulated structure from $\C$.

\subsection{A relative structure}
Remind that a triangulated category $\C$ admits a natural extriangulated structure $(\C,\mathbb{E},\mathfrak{s})$, see Proposition \ref{prop_extri_to_tri}.
We will show that any extension-closed subcategory $\N\subseteq\C$ determines natural extriangulated structures on $(\C,\mathbb{E},\mathfrak{s})$.
The following proposition is a main result in this subsection.


\begin{proposition}\label{prop_relative_str}
Let $\N$ be a full subcategory of $\C$ which is closed under direct summands, isomorphisms and extensions.
For any objects $A,C\in\C$, we define subsets of $\mathbb{E}(C,A)$ as follows.
\begin{enumerate}
\item[\textnormal{(1)}]
A subset $\mathbb{E}^L_\N(C,A)$ is defined as the set of morphisms $h\colon C\to A[1]$ satisfying the condition:
\begin{enumerate}
\item[{\rm (Lex)}] For any morphism $N\xto{x}C$ with $N\in\N$, $h\circ x$ factors through an object in $\N[1]$.
\end{enumerate}
\item[\textnormal{(2)}]
A subset $\mathbb{E}^R_\N(C,A)$ is defined as the set of morphisms $h\colon C\to A[1]$ satisfying the condition:
\begin{enumerate}
\item[{\rm (Rex)}] For any morphism $A\xto{y}N$ with $N\in\N$, $y\circ h[-1]$ factors through an object in $\N[-1]$.
\end{enumerate}
\end{enumerate}
The conditions ${\rm (Lex)}$ and ${\rm (Rex)}$ for an $\mathfrak{s}$-triangle $A\overset{f}{\lra}B\overset{g}{\lra}C\overset{h}{\dashrightarrow}$ are respectively understood through the following morphisms of $\mathfrak{s}$-triangles:
\[
\xymatrix{
N[-1]\ar[d]\ar[r]^{[\N]}&A\ar@{=}[d]\ar[r]^{f'}&B'\ar[d]\ar[r]^{g'}&N\ar[d]^x&A\ar[r]^{f}\ar[d]_y\ar@{}[rd]|{\rm (wPO)}&B\ar[d]\ar[r]^{g}&C\ar@{=}[d]\ar[r]^h&A[1]\ar[d]\\
C[-1]\ar[r]^{h[-1]}&A\ar[r]^{f}&B\ar[r]^g&C\ar@{}[lu]|{\rm (wPB)}&N\ar[r]^{f''}&B''\ar[r]^{g''}&C\ar[r]^{[\N]}&N[1]
}
\]
where the arrows labeled with $[\N]$ factor through objects in $\N$.
Then, both $\mathbb{E}^L_\N$ and $\mathbb{E}^R_\N$ give rise to closed subfunctors of $\mathbb{E}$.
In particular, putting $\mathbb{E}_\N:=\mathbb{E}^L_\N\cap\mathbb{E}^R_\N$, we have three extriangulated structures
\[(\C,\mathbb{E}^L_\N,\mathfrak{s}^L_\N),\quad (\C,\mathbb{E}^R_\N,\mathfrak{s}^R_\N),\quad (\C,\mathbb{E}_\N,\mathfrak{s}_\N)
\]
which are relative to $(\C,\mathbb{E},\mathfrak{s})$.
Here $\mathfrak{s}_\N$ is a restriction of $\mathfrak{s}$ to $\mathbb{E}_\N$ and other undefined symbols are used in similar meanings.
\end{proposition}

Before proving this proposition, we observe several extremal cases to understand what the conditions {\rm (Lex)} and {\rm (Rex)} mean.

\begin{example}\label{ex_relative_str_thick}
If a given subcategory $\N$ is thick in the triangulated category $(\C,\mathbb{E},\mathfrak{s})$, then we have $\mathbb{E}=\mathbb{E}^L_\N=\mathbb{E}^R_\N$.
In particular, the relative structure $(\C,\mathbb{E}_\N,\mathfrak{s}_\N)$ coincides with the original triangulated structure $(\C,\mathbb{E},\mathfrak{s})$.
\end{example}
\begin{proof}
Since $\N$ is thick, the conditions {\rm (Rex)} and {\rm (Lex)} in Proposition \ref{prop_relative_str} always hold for any $\mathfrak{s}$-triangle.
\end{proof}

\begin{example}\label{ex_relative_str_rigid1}
Let $\N$ be a rigid subcategory of $\C$, namely, $\mathbb{E}(N,N')=0$ for any $N,N'\in\N$.
Then, $\N$ is projective and injective in $(\C,\mathbb{E}_\N,\mathfrak{s}_\N)$ in the sense that $\mathbb{E}_\N(\N,-)=\mathbb{E}_\N(-,\N)=0$ (see \cite[Def. 3.23]{NP19}).
The structure $(\C,\mathbb{E}_\N,\mathfrak{s}_\N)$ is maximal with respect to the above property.
\end{example}
\begin{proof}
We shall show that any $N\in\N$ is projective in $(\C,\mathbb{E}_\N,\mathfrak{s}_\N)$.
Consider an $\mathfrak{s}_\N$-triangle $A\overset{f}{\lra} B\overset{g}{\lra} N\overset{h}{\dashrightarrow}$ which corresponds to an element in $\mathbb{E}_\N(N,A)$.
The condition ${\rm (Lex)}$ shows that $h$ factors through an object in $\N[1]$.
By the rigidity of $\N$, we get $h=0$ and the $\mathfrak{s}_\N$-triangle splits, which shows the projectivity of $N$.

To show the maximality of $(\C,\mathbb{E}_\N,\mathfrak{s}_\N)$, let us consider a relative structure $(\C,\mathbb{F},\mathfrak{t})$ with $\mathbb{F}(\N,-)=\mathbb{F}(-,\N)=0$.
It suffices to show that any $\mathfrak{t}$-triangle $A\overset{f}{\lra} B\overset{g}{\lra} C\overset{h}{\dashrightarrow}$ corresponding to an element in $\mathbb{F}(C,A)$ satisfies the condition ${\rm (Lex)}$.
For any morphism $N\xto{x}C$ from $N\in\N$, by the projectivity, $h\circ x$ should be zero.
Combining the dual argument, we finish the proof.
\end{proof}

\begin{example}\label{ex_relative_str_rigid}
Let $\X$ be a rigid contravariantly finite subcategory of $\C$ and put $\N:=\{C\in\C\mid (\X,C)=0\}$.
Since $\N$ is extension-closed in $\C$, we can consider the relative structure $(\C,\mathbb{E}^R_\N,\mathfrak{s}^R_\N)$.
In this case, $\mathbb{E}^R_\N$ can be understood as
\[
\mathbb{E}^R_\N(C,A)=\{h\in\mathbb{E}(C,A)\mid h\circ y=0\text{\ for all\ }y\colon X\to C\text{\ with\ }X\in\X\}.
\]
Such a structure is investgated in \cite{JS22} in terms of Grothendieck groups.
The idea goes back to \cite{PPPP19} where it was introduced to investigate relations between $\mathbf{g}$-vectors and cluster categories.
The above equality will be checked in Example \ref{ex_relative_str_rigid2}.
\end{example}

The proof of Proposition \ref{prop_relative_str} will be completed after a few auxiliary lemmas.
To this end, we need to consider the class $\mathscr{S}'_\N$ of morphisms $g\in\Mor\C$ which induces a triangle $A\overset{f}{\lra}B\overset{g}{\lra}C\overset{h}{\lra}A[1]$ with $\overline{f}=0$ and $\overline{h}=0$.
Useful characterizations for $\mathscr{S}'_\N$ will be given in Lemma \ref{lem_phantom}.
Remind that we write $f\in [\N]$ if $f$ factors through an object in $\N$.


\begin{lemma}\label{lem_extension_closed}
Let $A\overset{f}{\lra} B\overset{g}{\lra} C\overset{h}{\lra} A[1]$ be a triangle in $\C$.
Then the following assertions hold.
\begin{enumerate}
\item[{\rm (1)}]
If $f\in [\N]$  and $C\in\N$, then $B\in\N$.
\item[{\rm (2)}]
If $h\in [\N]$ and $B\in\N$, then $C\in\N$.
\end{enumerate}
\end{lemma}
\begin{proof}
We only check the assertion (1), since (2) can be checked dually.
By the assumption, $f$ admits a factorization $f\colon A\xto{f_1}N\xto{f_2}B$.
Taking a homotopy pushout of $f$ along $f_1$, we get the following commutative diagram made of triangles.
\[
\xymatrix{
A\ar[r]^f\ar[d]_{f_1}\ar@{}[rd]|{\rm (wPO)}&B\ar[r]^g\ar[d]^b&C\ar[r]^h\ar@{=}[d]&A[1]\ar[d]\\
N\ar[r]&B'\ar[r]&C\ar[r]&N[1]
}
\]
Since $f$ factors through $f_1$, $b$ is a section.
If $C\in\N$, since $\N$ is closed under taking extensions and direct summands, we get $B'\in\N$ and $B\in\N$.
\end{proof}

\begin{lemma}\label{lem_phantom}
For a triangle $A\overset{f}{\lra} B\overset{g}{\lra} C\overset{h}{\lra} A[1]$ in $\C$, the following are equivalent.
\begin{enumerate}
\item[{\rm (i)}]
$\overline{h}=0$ in $\overline{\C}$.
\item[{\rm (ii)}]
The morphism $g$ admits a factorization $g=g_2\circ g_1$ with $\cone(g_1)\in\N$ and $g_2$ being a retraction.
\item[{\rm (iii)}]
The morphism $\overline{g}$ is epic in $\overline{\C}$.
\end{enumerate}
Dually the following assertions are equivalent.
\begin{enumerate}
\item[{\rm (i')}]
$\overline{f}=0$ in $\overline{\C}$.
\item[{\rm (ii')}]
The morphism $g$ admits a factorization $g=g_2\circ g_1$ with $\cocone(g_2)\in\N$ and $g_1$ being a section.
\item[{\rm (iii')}]
The morphism $\overline{g}$ is monic in $\overline{\C}$.
\end{enumerate}
In particular, $g$ belongs to $\mathscr{S}'_\N$ if and only if $\overline{g}$ is monic and epic in $\overline{\C}$.
\end{lemma}
\begin{proof}
(i) $\Rightarrow$ (ii): 
By the assumption, $h$ admits a factorization $h\colon C\xto{h_1}N\xto{h_2}A[1]$ with $N\in\N$.
Taking a homotopy pullback of $h$ along $h_2$ yields the commutative diagram below
\begin{equation}\label{diag_phantom}
\xymatrix{
N[-1]\ar[d]_{h_2[-1]}\ar[r]&B\ar[r]^{\widetilde{g}}\ar@{=}[d]&\widetilde{C}\ar[r]^{\widetilde{h}}\ar[d]_c&N\ar[d]^{h_2}\\
A\ar[r]^f&B\ar[r]^g&C\ar[r]^h&A[1]\ar@{}[lu]|{\rm (wPB)}
}
\end{equation}
where all rows are triangles.
Since $h$ factors through $h_2$, $c$ is a retraction.
Thus we have a desired factorization $g=c\circ \widetilde{g}$.

(ii) $\Rightarrow$ (i): Due to the octahedral axiom, the assertion can be easily checked.

(i) $\Leftrightarrow$ (iii):
The implication (iii) $\Rightarrow$ (i) follows from $h\circ g=0$.
To show the converse, we assume that $h$ has a factorization $h\colon C\xto{h_1}N\xto{h_2}A[1]$ with $N\in\N$ and let $C\xto{x}X$ be a morphism with $\overline{x}\circ\overline{g}=0$.
By taking a homotopy pullback of $h$ along $h_2$, we obtain the same diagram as (\ref{diag_phantom}).
Our assumption shows that $x\circ g$ admits a factorization $x\circ g\colon B\xto{y_1}N'\xto{y_2}X$ with $N'\in\N$.
Taking a homotopy pushout of $\widetilde{g}$ along $y_1$, we have the following morphism of triangles.
\[
\xymatrix{
N[-1]\ar@{=}[d]\ar[r]&B\ar@{}[rd]|{\rm (wPO)}\ar[d]_{y_1}\ar[r]^{\widetilde{g}}&\widetilde{C}\ar[r]^{\widetilde{h}}\ar[d]^a&N\ar@{=}[d]\\
N[-1]\ar[r]&N'\ar[r]&N''\ar[r]&N
}
\]
Since $\N$ is extension-closed, we get $N''\in\N$.
Then we get $x\circ g=x\circ c\circ \widetilde{g}=y_2\circ y_1$ which induces a morphism $b\colon N''\to X$ with $b\circ a=x\circ c$.
The morphism $b$ is denoted by the dotted arrow in the commutative diagram below.
\[
\xymatrix{
B\ar@{}[rd]|{\rm (wPO)}\ar[r]^{\widetilde{g}}\ar[d]_{y_1}&\widetilde{C}\ar[d]^a\ar@/^10pt/[rdd]^-{x\circ c}&\\
N'\ar[r]\ar@/_10pt/[rrd]_{y_2}&N''\ar@{.>}[rd]^-{^{\exists}b}&\\
&&X
}
\]
We denote by $C\xto{c'}\widetilde{C}$ the section corresponding to the retraction $c$.
It turns out that a factorizaion $x\colon C\xto{c'}\widetilde{C}\xto{a}N''\xto{b}X$ exists and shows $\overline{x}=0$ in $\overline{\C}$.

The remaining assertions can be proved in a dual manner.
\end{proof}

\begin{remark}
We can find similar arguments to the above in the literature.
In fact, it is an analogue of well-known characterizations of pure-exact triangles, see \cite{Bel00, Kra00}.
\end{remark}

We are ready to prove Proposition \ref{prop_relative_str}.

\begin{proof}[Proof of Proposition \ref{prop_relative_str}]
It is straightforward that both $\mathbb{E}^L_\N$ and $\mathbb{E}^R_\N$ are biadditive subfunctors of $\mathbb{E}$.
To confirm the closedness of $\mathbb{E}^L_\N$,
let us consider elements $h\in\mathbb{E}^L_\N(C,A)$ and $b''\in\mathbb{E}^L_\N(B'',B)$ which correspond to triangles $C[-1]\overset{h[-1]}{\lra}A\overset{f}{\lra}B\overset{g}{\lra}C$ and $B''[-1]\overset{b''[-1]}{\lra}B\overset{b}{\lra}B'\overset{b'}{\lra}B''$, respectively.
We shall show that the composed morphism $x:=b\circ f$ is an $\mathfrak{s}_\N^L$-inflation.
By the octahedral axiom, the composition $x$ yields the following commutative diagram of solid arrows
\[
\xymatrix{
&C[-1]\ar[r]^{h[-1]}\ar[d]&A\ar[r]^f\ar@{=}[d]&B\ar[r]^g\ar[d]^b&C\ar[d]^c\\
N[-1]\ar@{..>}[r]^w&C'[-1]\ar[r]^{z[-1]}\ar[d]&A\ar[r]^x\ar[d]_f&B'\ar[r]^y\ar@{=}[d]&C'\ar[d]^{c'}\\
&B''[-1]\ar[r]^{b''[-1]}\ar@{=}[d]&B\ar[r]^b\ar[d]&B'\ar[r]^{b'}\ar[d]&B''\ar@{=}[d]\\
&B''[-1]\ar[r]&C\ar[r]&C'\ar[r]&B''
}
\]
in which all rows are triangles.
Let $w\colon N[-1]\to C'[-1]$ be a morphism with $N\in\N$.
We have to only show that the composed morphism $\phi:=z[-1]\circ w$ factors through an object in $\N$.
Note that $b''\in\mathbb{E}^L_\N(B'',B)$ and $f\circ \phi=b''[-1]\circ c'[-1]\circ w$ factors through an object $N'\in\N$.
Thus, we get a factorization $f\circ \phi\colon N[-1]\xto{\psi_2}N'\xto{\psi_1}B$.
Taking a homotopy pullback of $f$ along $\psi_1$ yields a morphism between the lower two triangles below in the commutative diagram (\ref{diag_relative_str}).
\begin{equation}\label{diag_relative_str}
\xymatrix{
N'[-1]\ar[r]\ar@{=}[d]&N''[-1]\ar[r]^{h''[-1]}\ar[d]_{\phi'_1}&N[-1]\ar[r]^{\psi_2}\ar[d]^{\phi_1}&N'\ar@{=}[d]\\
N'[-1]\ar[r]_{g'[-1]}\ar[d]&C[-1]\ar[r]_{h'[-1]}\ar@{=}[d]&A'\ar[r]_{f'}\ar[d]^{\phi_2}\ar@{}[lu]|{\rm (wPB)}&N'\ar[d]^{\psi_1}\\
B[-1]\ar[r]_{g[-1]}&C[-1]\ar[r]_{h[-1]}&A\ar[r]_f&B\ar@{}[lu]|{\rm (wPB)}
}
\end{equation}
By the property of homotopy pullback, $\phi:N[-1]\to A$ is factorized as $\phi=\phi_2\circ\phi_1$.
The morphism between upper triangles in (\ref{diag_relative_str}) is obtained by taking a homotopy pullback of $h'[-1]$ along $\phi_1$.
Since $h\in\mathbb{E}^L_\N(C,A)$ and $N''\in\N$, the morphism $\phi_2\circ h'[-1]\circ\phi'_1=\phi\circ h''[-1]$ factors through an object in $\N$, namely, $\overline{\phi}\circ\overline{h''[-1]}=0$ in $\overline{\C}$.
By Applying Lemma \ref{lem_phantom} to the first row in (\ref{diag_relative_str}), we conclude that $\overline{h''[-1]}$ is an epimorphism in $\overline{\C}$ and $\overline{\phi}=0$.

The assertion for $\mathbb{E}^R_\N$ can be checked dually.
Hence, $\mathbb{E}_\N$ is also a closed subfunctor.
\end{proof}

We push a bit more investigations on the extriangulated category $(\C,\mathbb{E}_\N,\mathfrak{s}_\N)$.
First, although the subcategory $\N$ is extension-closed in the triangulated category $\C$,
it is moreover a thick subcategory with respect to the strucutre $(\C,\mathbb{E}_\N,\mathfrak{s}_\N)$.

\begin{corollary}\label{cor_thick}
The subcategory $\N$ is a thick subcategory of the extriangulated category $(\C,\mathbb{E}_\N,\mathfrak{s}_\N)$.
\end{corollary}
\begin{proof}
Let $A\overset{f}{\longrightarrow} B\overset{g}{\longrightarrow} C\overset{h}{\dashrightarrow}$ be an $\mathfrak{s}_\N$-triangle of $(\C,\mathbb{E}_\N,\mathfrak{s}_\N)$.
Suppose that $A$ and $B$ belong to $\N$.
By definition of $\mathfrak{s}_\N$-triangles, $A\in\N$ implies that $C\xto{h} A[1]$ factors through an object in $\N$.
By Lemma \ref{lem_extension_closed}, we get $C\in\N$.
Similarly, $B,C\in\N$ implies $A\in\N$.
As $\N$ is extension-closed from the beginning, the assertion have been proved.
\end{proof}

Secondly, we mention the {\rm (WIC)} property of the extriangulated category $(\C,\mathbb{E}_\N,\mathfrak{s}_\N)$, see \cite[Condition 5.8]{NP19}.
In \cite[Thm. C, Prop. 2.7]{Kla22}, it was shown that {\rm (WIC)} for any extriangulated category $(\C,\mathbb{E}, \mathsf{s})$ is equivalent to $\C$ being weakly idempotent complete.
Since every triangulated category is weakly idempotent complete, we conclude as below.

\begin{proposition}\label{prop_WIC}
The extriangulated categories $(\C,\mathbb{E}_\N,\mathfrak{s}_\N), (\C,\mathbb{E}^L_\N,\mathfrak{s}^L_\N)$ and $(\C,\mathbb{E}^R_\N,\mathfrak{s}^R_\N)$ satisfy {\rm (WIC)}.
\end{proposition}

We remark that (WIC) is particularly important in various contexts (e.g., the localization by Hovey twin cotorsion pairs is formulated under (WIC) \cite[Sect. 5 and 6]{NP19}).

\subsection{A construction of multiplicative systems}
This subsection is devoted to prove that the class $\Sn$ forms a multiplicative system in the ideal quotient $\C/[\N]$ and define the localization of $\C$ with respect to an extension-closed subcategory $\N$.

Recall from the previous subsection that a given pair of a triangulated category $(\C,\mathbb{E},\mathfrak{s})$ and an extension-closed subcategory $\N$ induces a pair of an extriangulated category $(\C,\mathbb{E}_\N,\mathfrak{s}_\N)$ and a thick subcategory $\N$.
Thus we can consider three classes $\L,\R$ and $\Sn$ of morphisms, see Definition \ref{def_Sn_from_thick}.
Let us start with easy observations.

\begin{lemma}\label{lem_Sn_from_extension-closed}
For the above $(\C,\mathbb{E}_\N,\mathfrak{s}_\N)$ and $\N$, the following hold.
\begin{itemize}
\item[{\rm (1)}]
$\L$ coincides with the class of morphisms $f\in\Mor\C$ which is embedded in a triangle $A\overset{f}{\lra}B\overset{g}{\lra}N\overset{h}{\lra}A[1]$ with $N\in\N$ and $\overline{h[-1]}=0$.
\item[{\rm (2)}]
$\R$ coincides with the class of morphisms $g\in\Mor\C$ which is embedded in a triangle $N\overset{f}{\lra}B\overset{g}{\lra}C\overset{h}{\lra}N[1]$ with $N\in\N$ and $\overline{h}=0$.
\end{itemize}
\end{lemma}
\begin{proof}
We only check the first assertion.
Let $f\in\L$ be given, namely, there exists an $\mathfrak{s}_\N$-triangle $A\overset{f}{\lra}B\overset{g}{\lra}N\overset{h}{\dashrightarrow}$ such that $N\xto{h}A[1]$ belongs to $h\in\mathbb{E}^L_\N(N,A)$ and $N\in\N$.
The fact $N\in\N$ shows that $h[-1]$ factors through an object in $\N$.
The converse is obvious.
\end{proof}

The following proposition provides a more explicit form of $\Sn$.
Recall that $\mathscr{S}'_\N$ is the class of morphisms $g\in\Mor\C$ appearing in a triangle $A\overset{f}{\lra}B\overset{g}{\lra}C\overset{h}{\lra}A[1]$ with $\overline{f}=0$ and $\overline{h}=0$.

\begin{proposition}\label{prop_description_of_Sn}
We have an equality $\Sn=\mathscr{S}'_\N$.
\end{proposition}

To prove the proposition, we use the following useful descriptions of $\mathscr{S}'_\N$ which matches well with the extriangulated localizations in Theorem \ref{Thm_Mult_Loc}.

\begin{lemma}\label{lem_RL}
Consider the following classes of morphisms in $\C$.
\begin{itemize}
\item[-] $\L_{\mathsf{sp}}:={\rm the\ class\ of\ sections\ belonging\ to\ }\L$.
\item[-] $\R_{\mathsf{sp}}:={\rm the\ class\ of\ retractions\ belonging\ to\ }\R$.
\end{itemize}
Then, we have $\mathscr{S}'_\N=\R_{\mathsf{sp}}\circ\L=\R\circ\L_{\mathsf{sp}}$.
\end{lemma}
\begin{proof}
Let $A\overset{f}{\lra} B\overset{s}{\lra} C\overset{g}{\lra} A[1]$ be a triangle in $\C$ and assume $s\in\mathscr{S}'_\N$.
By the definition of $\mathscr{S}'_\N$, $g$ admits a factorization $g:C\xto{g_1}N\xto{g_2}A[1]$ with $N\in\N$.
Considering a homotopy pullback of $g$ along $g_2$, thanks to the octahedral axiom, we obtain the following commutative diagram.
\begin{equation}\label{diag_lem_RL}
\xymatrix{
&&C[-1]\ar[r]\ar[d]_0&A\ar[d]^a\\
&&N_C\ar[d]_i\ar@{=}[r]&N_C\ar[d]\\
N[-1]\ar[d]\ar[r]^{\widetilde{f}}&B\ar@{=}[d]\ar[r]^{\widetilde{s}}&\widetilde{C}\ar[r]^{\widetilde{g}}\ar[r]\ar[d]_{p}&N\ar[d]^{g_2}\\
A\ar[r]^f&B\ar[r]^s&C\ar[r]^g&A[1]\ar@{}[lu]|{\rm (wPB)}
}
\end{equation}
We shall show that the above factorization $s=p\circ \widetilde{s}$ is a desired one.
First, we remind that {\rm (wPB)} induces a triangle
\[A\overset{-\delta}{\lra} \widetilde{C}\overset{\binom{-\widetilde{g}}{p}}{\lra} N\oplus C\overset{(g_2\, g)}{\lra} A[1]
\]
with $\delta=\widetilde{s}\circ f=i\circ a$.
The lower left square in \eqref{diag_lem_RL} and our assumption $f\in[\N]$ guarantees $\widetilde{f}\in [\N]$.
We have thus checked $\widetilde{s}\in\L$ by Lemma \ref{lem_Sn_from_extension-closed}(1).
It remains to show $p\in\R_{\mathsf{sp}}$.
Since $g$ factors through $g_2$, we have that $p$ is a retraction and $i$ is a section.
Again, by our assumption $f\in [\N]$ and $\widetilde{s}\circ f=i\circ a$, we have $a\in [\N]$.
Applying Lemma \ref{lem_extension_closed}(1) to the forth column in \eqref{diag_lem_RL}, we obtain $N_C\in\N$ and $p\in\R_{\mathsf{sp}}$.
It turns out that $\mathscr{S}'_\N\subseteq\R_\mathsf{sp}\circ\L$ is true.

To show the converse, we consider a composed morphism $s=p\circ \widetilde{s}\colon B\xto{\widetilde{s}}\widetilde{C}\xto{p}C$ with $\widetilde{s}\in\L$ and $p\in\R_{\mathsf{sp}}$.
By the octahedral axiom, we have a commutative diagram of the shape same as \eqref{diag_lem_RL}.
Now we have to only check $f,g \in [\N]$.
The lower left square in \eqref{diag_lem_RL} shows $\widetilde{f}=f\circ (g_2[-1])$.
We know $\overline{g_2[-1]}$ is epic in $\overline{\C}$ by Lemma~\ref{lem_phantom}.
Thus the assumption $\widetilde{f}\in [\N]$ implies $f\in [\N]$.
Next, the square ${\rm (wPB)}$ shows $\overline{g\circ p} = \overline{g_2\circ \widetilde{g}}=0$.
As $\overline{p}$ is epic, we get $g\in [\N]$.
As a conclusion, we obtain $s\in\mathscr{S}'_\N$ and $\mathscr{S}'_\N = \R_\mathsf{sp}\circ\L$.
By a dual argument, another equality can be verified.
\end{proof}

Now we are in position to prove Proposition \ref{prop_description_of_Sn}.

\begin{proof}[Proof of Proposition \ref{prop_description_of_Sn}]
Since $\Sn$ consists of finite compositions of morphisms belonging to $\L$ or $\R$, by Lemma \ref{lem_RL}, the containment $\Sn\supseteq \mathscr{S}'_\N$ is obvious.
Conversely, due to Lemma \ref{lem_phantom},
any morphism in $\L\cup\R$ becomes monic and epic in $\overline{\C}$ and belongs to $\mathscr{S}'_\N$.
Hence we have $\Sn=\mathscr{S}'_\N$.
\end{proof}

\begin{example}
We list some typical examples of the above $\Sn$ appearing in well-known constructions.

\begin{enumerate}
\item
For a triangulated category $\C$ and a thick subcategory $\N$ of $\C$, we consider the Verdier quotient $\mathsf{Ver}\colon \C\to\C/\N$.
The class $\mathscr{S}:=\{s\in\Mor\C\mid \mathsf{Ver}(s)\textnormal{\ is an isomorphism}\}$ of morphisms coincides with the above $\Sn$ (Theorem \ref{cor_tri}).
\item
Let $\X$ be a contravariantly finite rigid subcategory of a triangulated category $\C$ and consider the associated cohomological functor $H:=(\X, -)\colon \C\to\mod\X$ to the category of finitely presented contravariant functors.
Note that $\N:=\Ker H$ is extension-closed in $\C$ and the extriangulated category $(\C,\mathbb{E}_\N,\mathfrak{s}_\N)$ exists.
Then, we have an equality $\Sn=\{s\in\Mor\C\mid H(s)\textnormal{\ is an isomrophism}\}$ (see Example \ref{ex_relative_str_rigid2} and Section 5.3.2).
\item
For a triangulated category $\C$ and a cotorsion pair $(\U,\V)$ of $\C$, we consider the associated cohomological functor $H\colon \C\to\H/[\W]$, where $\H/[\W]$ is the heart of $(\U,\V)$ introduced in \cite{Nak11, AN12}.
Then, putting $\N:=\Ker H$, we have an equality $\Sn=\{s\in\Mor\C\mid H(s)\textnormal{\ is an isomorphism}\}$ (Proposition \ref{prop_kernel_heart}).
\end{enumerate}
\end{example}

As a benefit of the description of $\Sn$ in Lemma \ref{lem_RL}, we can prove that $\Sn$ forms multiplicative system in $\overline{\C}$.

\begin{lemma}
We have an equality $\N=\N_{\Sn}$.
\end{lemma}
\begin{proof}
Due to Lemma \ref{lem_Sn_from_extension-closed}, for any $X\in\N$, the zero morphisms $0\to X$ and $X\to 0$ belong to $\Sn$. Thus we get $X\in\N_{\Sn}$.
To show the converse, we consider $X\in\N_{\Sn}$, namely, the zero morphisms $0\to X$ and $X\to 0$ belong to $\Sn$.
Since $\Sn=\R_{\mathsf{sp}}\circ\L$ is true,
we have a
triangle $0\xto{l}B\to N\to 0$ for some $N\in\N$ and a retraction $p\colon B\to X$, so
that $0\to X$ factors as $0\to B\xto{p}X$.
Thus we have $B\in\N$ and $X\in \N$.
\end{proof}

The class $\overline{\Sn}$ forms a multiplicative system in $\overline{\C}$ as we now verify.
Recall that, for any class $\mathscr{S}$, the symbol $\overline{\mathscr{S}}^*$ denotes the closure of $\overline{\mathscr{S}}$ with respect to the compositions with isomorphisms in $\overline{\C}$.

\begin{proposition}\label{prop_localization}
Let $\C$ and $\N$ be the above.
Then, $\overline{\Sn}=\overline{\Sn}^*$ holds and the class $\overline{\Sn}$ is a multiplicative system in the ideal quotient $\overline{\C}=\C/[\N]$.
We define the \emph{localization of $\C$ with respect to the subcategory $\N$} as the Gabriel-Zisman localization $Q\colon \C\to\C[\Sn^{-1}]$.
\end{proposition}
\begin{proof}
By Lemma \ref{lem_phantom}, morphisms $f\in\Sn$ are characterized by the property that $\overline{f}$ is monic and epic in $\overline{\C}$.
Hence we have $\overline{\Sn}=\overline{\Sn}^*$.

(MS0): If composable morphisms $\overline{f}$ and $\overline{g}$ are epic in $\overline{\C}$, then so is $\overline{g}\circ\overline{f}$.
Thus, due to Lemma \ref{lem_phantom}, it follows that $\Sn$ is closed under composition.
Obviously, $\Sn$ is closed under taking finite direct sums.

(MS2): Consider a sequence $A\xto{f}B\xto{s}C$ in $\C$ with $\overline{s}\circ\overline{f}=0$ and $s\in\overline{\Sn}$.
Then, since $\overline{s}$ is monic, $f$ factors through an object in $\N$.

(MS1): Consider morphisms $X\xleftarrow{x}B\xto{s}C$ in $\C$ and assume $s\in\Sn$.
Due to Lemma \ref{lem_RL}, we have a factorization $s=p\circ \widetilde{s}\colon B\xto{\widetilde{s}}\widetilde{C}\xto{p}C$ with $\widetilde{s}\in\L$ and $p\in\R_{\mathsf{sp}}$ as in \eqref{diag_lem_RL}.
Taking weak pushout of $\widetilde{s}$ along $x$  yields the following morphism of triangles,
\[
\xymatrix{
N[-1]\ar@{=}[d]\ar[r]&B\ar@{}[rd]|{\rm (wPO)}\ar[d]_x\ar[r]^{\widetilde{s}}&\widetilde{C}\ar[d]^{x'}\ar[r]&N\ar@{=}[d]\\
N[-1]\ar[r]&X\ar[r]^t&X'\ar[r]&N
}
\]
where $N\in\N$.
The weak pushout square ${\rm (wPO)}$ gives rise to a desired one.
Actually, since $\widetilde{s}\in\L$, we have $t\in\L$ together with $t\circ x=x'\circ \widetilde{s}$.
In addition, since $p\in\R_{\mathsf{sp}}$, there exists the corresponding section $C\xto{p'}\widetilde{C}$ such that $\overline{p'\circ p} = \overline{\id_{\widetilde{C}}}$ holds in $\overline{\C}$.
We have thus obtained a desired commutativity $\overline{t\circ x}=\overline{x'\circ \widetilde{s}}=\overline{x'\circ p'\circ p\circ \widetilde{s}}$ in $\overline{\C}$.

The remaining conditions can be checked dually.
\end{proof}

By the argument so far, we have the Gabriel-Zisman localization $\overline{\C}\to\overline{\C}[(\overline{\Sn})^{-1}]$ admitting left and right fractions.
Keeping in mind that, by the universality, there exists an equivalence $\overline{\C}[(\overline{\Sn})^{-1}]\simeq\C[\Sn^{-1}]$, we put $\widetilde{\C}_\N:=\C[\Sn^{-1}]$ and call the natural functor $Q\colon \C\to\widetilde{\C}_\N$ \emph{the localization of $\C$ with respect to the full subcategory $\N$}.
In particular, any morphism $\alpha\colon A\to B$ of $\widetilde{\C}_\N$ can be represented by the commutative diagram in $\overline{\C}$:
\[
\xymatrix@R=12pt@C=16pt{
&B'&\\
A\ar[ru]^{\overline{f}}&&B\ar[lu]_{\overline{s}}\\
&\ar[lu]^{\overline{t}}A'\ar[ru]_{\overline{g}}&
}
\]
with $s,t\in\Sn$. More precisely,  we get $\alpha=Q(s)^{-1}\circ Q(f)=Q(g)\circ Q(t)^{-1}$.
Thanks to such descriptions of morphisms, we can confirm nicer properties of the localization $\widetilde{\C}_\N=\C[\Sn^{-1}]$ with respect to $\N$.
In advance, we give characterizations of objects and morphisms which  are zero in $\widetilde{\C}_\N$.

\begin{corollary}\label{cor_phantom}
The following assertions hold.
\begin{enumerate}
\item[{\rm (1)}]
For any morphism $f\in\Mor\C$, $Q(f)=0$ if and only if $f$ factors through an object in $\N$.
\item[{\rm (2)}]
For any object $A\in\C$, $Q(A)=0$ if and only if $A\in\N$.
\end{enumerate}
\end{corollary}
\begin{proof}
(1) Let $f\colon A\to B$ be a morphism with $Q(f)=0$.
Then, there exists a morphism $s\colon B\to B'$ of $\Sn$ such that $\overline{s}\circ\overline{f}=0$.
Since $\overline{s}$ is monic, we have $\overline{f}=0$.
The converse is obvious.

(2) By the assertion (1), $Q(A)=0$ implies $\overline{\id_A}=0$. Hence $A\in\N$.
\end{proof}

The following provides characterizations of monomorphisms, epimorphisms and isomorphisms in $\widetilde{\C}_\N$.

\begin{proposition}\label{prop_saturated}
Let $A\overset{f}{\lra}B\overset{g}{\lra}C\overset{h}{\lra}A[1]$ be a triangle in $\C$.
Then the following are equivalent.
\begin{enumerate}
\item[{\rm (i)}]
$\overline{h}=0$ in $\overline{\C}$.
\item[{\rm (ii)}]
$\overline{g}$ is an epimorphism in $\overline{\C}$.
\item[{\rm (iii)}]
$Q(g)$ is an epimorphism in $\widetilde{\C}_\N$.
\end{enumerate}
Dually the following assertions are equivalent.
\begin{enumerate}
\item[{\rm (i')}]
$\overline{f}=0$ in $\overline{\C}$.
\item[{\rm (ii')}]
$\overline{g}$ is a monomorphism in $\overline{\C}$.
\item[{\rm (iii')}]
$Q(g)$ is a monomorphism in $\widetilde{\C}_\N$.
\end{enumerate}
In particular, $Q(g)$ is an isomorphism if and only if $g\in\Sn$ if and only if $\overline{g}$ is epic and monic.
\end{proposition}
\begin{proof}
We will only check the former assertions, since the latter ones can be checked in a dual manner.

(i) $\Leftrightarrow$ (ii): It has been already confirmed in Lemma \ref{lem_phantom}.

(ii) $\Rightarrow$ (iii): 
Let $x\colon C\to X$ be a morphism in $\C$ with $Q(x\circ g)=0$.
Corollary \ref{cor_phantom} shows $\overline{x}\circ \overline{g}=0$.
As $\overline{g}$ is epic, we have $\overline{x}=0$ and $Q(x)=0$.

(iii) $\Rightarrow$ (ii):
Let $x\colon C\to X$ be a morphism in $\C$ with $\overline{x\circ g}=0$.
Then, $Q(g\circ x)=0$ implies $Q(x)=0$.
Again, Corollary \ref{cor_phantom} shows $\overline{x}=0$.

It remains to show the last equivalences, which can be extracted as follows.
\begin{align*}
Q(g) \text{\ is an isomorphism}
    &\,\implies \overline{g} \text{\ is epic and monic} &\text{ (by (iii) $\Rightarrow$ (ii) and (iii') $\Rightarrow$ (ii')) }\\
    &\iff \overline{h}=\overline{f}=0 &\text{ (by (ii) $\Leftrightarrow$ (i) and (ii') $\Leftrightarrow$ (i')) }\\
    &\iff g\in\mathscr{S}'_{\N}=\Sn &\text{ (by Proposition \ref{prop_description_of_Sn}) }\\
    &\,\implies Q(g) \text{\ is an isomorphism} &\text{ (by Definition) }
\end{align*}
We have thus finished the proof.
\end{proof}

The property {\rm (MR1)} directly follows from Proposition \ref{prop_saturated}.

\begin{corollary}\label{cor_MR1}
The multiplicative system $\Sn$ satisfies {\rm (MR1)}, namely, the $2$-out-of-$3$ with respect to compositions.
Moreover, $\overline{\Sn}$ satisfies the $2$-out-of-$3$ with respect to compositions.
\end{corollary}
\begin{proof}
Let $A\xto{s}B\xto{t}C$ be a sequence in $\C$ and assume $t\circ s, s\in\Sn$.
Then $Q(t)$ becomes isomorphism.
Proposition \ref{prop_saturated} shows $t\in\Sn$.
The other conditions can be easily checked.
\end{proof}

Since our aim is to impose a natural extriangulated structure on $\widetilde{\C}_\N$ in which any objects in $\N$ vanishes, the following properties of $\mathfrak{s}_\N$-triangles should be expected.

\begin{lemma}\label{lem_relative_str}
Let $A\overset{f}{\lra} B\overset{g}{\lra} C\overset{h}{\dashrightarrow}$ be an $\mathfrak{s}_\N$-triangle.
\begin{enumerate}
\item[\textnormal{(1)}]
If $A\in\N$, then $Q(g)$ is an isomorphism.
\item[\textnormal{(2)}]
If $C\in\N$, then $Q(f)$ is an isomorphism.
\end{enumerate}
\end{lemma}
\begin{proof}
We only prove (1).
By definition, $A\in\N$ implies that $h$ factors through an object in $\N$. Hence $g\in\Sn$ and so $Qg$ is an isomorphism by Proposition \ref{prop_saturated}.
\end{proof}

\subsection{A construction of exact functors}
Our main result in this section provides a way to impose a natural extriangulated structure on $\widetilde{\C}_\N$ which makes the localization $Q\colon \C\to\widetilde{\C}_\N$ to be an exact functor with an appropriate universality.
We remind the assumptions that the subcategory $\N$ is closed under taking extensions and direct summands are in play.

%

\begin{theorem}\label{thm_main1}
Let us consider the relative extriangulated category $(\C,\mathbb{E}_\N,\mathfrak{s}_\N)$ determined by the subcategory $\N$ and the associated localization $Q\colon (\C,\mathbb{E}_\N,\mathfrak{s}_\N)\to\widetilde{\C}_\N$.
Then, the multiplicative system $\overline{\Sn}$ satisfies the conditions {\rm (MR1), \ldots, (MR4)}.
In particular, $\widetilde{\C}_\N$ inherits an extriangulated structure from $\C$ and the functor $Q$ gives rise to an exact functor $(Q,\mu)$ which satisfies the universality:
For any exact functor $(F,\phi)\colon (\C,\mathbb{E}_\N,\mathfrak{s}_\N)\to (\D,\mathbb{F},\mathfrak{t})$ between extriangulated categories with $\N\subseteq\Ker F$, there uniquely exists an exact functor $(\widetilde{F},\widetilde{\phi})\colon \widetilde{\C}_\N\to \D$ with $(F,\phi)=(\widetilde{F},\widetilde{\phi})\circ (Q,\mu)$ as depicted below.
\[
\xymatrix{
(\C,\mathbb{E}_\N,\mathfrak{s}_\N)\ar[r]^{(Q,\mu)}\ar[d]_{(F,\phi)}&(\widetilde{\C}_\N,\widetilde{\mathbb{E}}_\N,\widetilde{\mathfrak{s}}_\N)\ar@{..>}[ld]^{(\widetilde{F},\widetilde{\phi})}\\
(\D,\mathbb{F},\mathfrak{t})&
}
\]
\end{theorem}
\begin{proof}
The conditions {\rm (MR1), (MR2)} have been checked in Proposition \ref{prop_localization} and Corollary \ref{cor_MR1}.

{\rm (MR3)}: 
Let us consider the following morphism $(a,b,c)$ of $\mathfrak{s}_\N$-triangles, and assume $a,c\in\Sn$.
\[
\xymatrix{
A\ar[r]^f\ar[d]_a\ar@{}[rd]&B\ar[r]^g\ar[d]^{b}&C\ar[r]^h\ar[d]^c&A[1]\ar[d]^{a[1]}\\
A'\ar[r]_{f'}&B'\ar[r]_{g'}&C'\ar[r]_{h'}\ar@{}[lu]&A'[1]
}
\]
We have to show that there is a choice of a morphism $b$ such that $b\in\Sn$ and $(a,b,c)$ still forms a morphism of $\mathfrak{s}_\N$-triangles.
First, by taking a weak pushout and a weak pullback successively, we get the following factorization,
\[
\xymatrix{
A\ar[r]^f\ar[d]_a\ar@{}[rd]|{\rm (wPO)}&B\ar[r]^g\ar[d]^{b'}&C\ar[r]^h\ar@{=}[d]&A[1]\ar[d]^{a[1]}\\
A'\ar[r]^{f_2}\ar@{=}[d]&B_2\ar[r]^{g_2}\ar[d]_{b''}&C\ar[r]^{h_2}\ar[d]^c&A'[1]\ar@{=}[d]\\
A'\ar[r]_{f'}&B'\ar[r]_{g'}&C'\ar[r]_{h'}\ar@{}[lu]|{\rm (wPB)}&A'[1]
}
\]
where the triplets $(a,b',\id_C)$ and $(\id_{A'},b'',c)$ are morphisms of $\mathfrak{s}_\N$-triangles, e.g. \cite[Lem. 2.4]{Eno21}.
Note that, as $h\in\mathbb{E}_\N(C,A)$ and $h'\in\mathbb{E}_\N(C',A')$, $h_2\in\mathbb{E}_\N(C,A')$ is also ture.

We shall show that, for the fixed morphisms $a$ and $c$, there is a choice of $b'$ which belongs to $\Sn$.
Due to Lemma \ref{lem_RL},  we have $a=a_2\circ a_1$ with $a_1\in\L, a_2\in\R_{\mathsf{sp}}$.
In particular, we have a splitting $\mathfrak{s}_\N$-triangle $N_A\overset{\beta}{\lra} A_1\overset{a_2}{\lra}A'\overset{0}{\dra}$ with $N_A\in\N$.
By taking weak pushouts of $f$ along $a=a_2a_1$ in succession, we establish the morphisms $(a_1,b'_1,\id_C)$ and $(a_2,b'_2,\id_C)$ between $\mathfrak{s}_\N$-triangles as in the following commutative diagram.
\begin{equation}\label{diag_thm_main1}
\xymatrix{
A\ar[r]^f\ar[d]_{a_1}\ar@{}[rd]|{\rm (wPO)}&B\ar[r]^g\ar[d]^{b'_1}&C\ar[r]^h\ar@{=}[d]&A[1]\ar[d]^{a_1[1]}\\
A_1\ar[d]_{a_2}\ar[r]^{f_1}\ar@{}[rd]|{\rm (wPO)}&B_1\ar[r]^{g_1}\ar[d]^{b'_2}&C\ar[r]^{h_1}\ar@{=}[d]&A_1[1]\ar[d]^{a_2[1]}\\
A'\ar[r]^{f_2}\ar@{..>}[d]_{0}&B_2\ar[r]^{g_2}\ar@{..>}[d]^{\alpha}&C\ar[r]^{h_2}&A'[1]\\
N_A[1]\ar@{=}[r]&N_A[1]\ar@{}[lu]|\circlearrowright&&
}
\end{equation}
Since $a_1\in\L$, we get $b'_1\in\L$ by the stability under taking weak pushout.
In contrast to this, we need more arguments to check the desired conditions for $b'_2$.
Since a weak pushout can be taken as a usual homotopy pushout, we complete $b'_2$ into a triangle $N_A\lra B_1\overset{b'_2}{\lra}B_2\overset{\alpha}{\lra} N_A[1]$ in the triangulated category $\C$.
Thus, Proposition \ref{prop_saturated} tells us $\overline{b'_2}$ is monic in $\overline{\C}$.
Again, thanks to Proposition \ref{prop_saturated}, we have to only show $\overline{\alpha}=0$.
Remind that the lower weak pushout square ${\rm (wPO)}$ induces a (usual) triangle in $\C$ which appears as the upper row in the following diagram.
\[
\xymatrix@C=42pt{
A_1\ar[r]^{\binom{-f_1}{a_2}}& B_1\oplus A'\ar[r]^{(b'_2\, f_2)}\ar[rd]_0& B_2\ar[r]^{-(h_1g_2)}\ar[d]^{\alpha}& A_1[1]\ar@{..>}[ld]^{\alpha'}\\
&&N_A[1]&
}
\]
Now, as we know $\alpha\circ f_2=0$ from the lower left square in \eqref{diag_thm_main1}, $\alpha\circ (b'_2\, f_2)=0$ holds and induces the dotted arrow $\alpha'$ in the above diagram.
As $h_1\in\mathbb{E}_\N(C,A_1)$, by the condition (Rex) in Proposition~\ref{prop_relative_str}, the composition $\alpha'\circ h_1$ factors through an object $N\in\N$.
It turns out that $b'_2$ belongs to $\Sn$.
We have thus verified $b':= b'_2\circ b'_1\in\Sn$.

In a dual manner, we can show that there exists a choice of $b''$ such that $b''\in\Sn$ and $(\id_{A'},b'',c)$ forms a morphism of $\mathfrak{s}_\N$-triangles.
As a conclusion, there is a desired choice of $b\in\Sn$ as a part of a morphism $(a,b,c)$ between $\mathfrak{s}_\N$-triangles.

{\rm (MR4)}: 
Since $\Sn$ is closed under compositions, it suffices to show that $\overline{f}\circ \overline{s}\circ \overline{f}'\in\overline{\M}_{\mathsf{inf}}$ for any $\mathfrak{s}_{\N}$-inflations $f,f'$ and any $s\in\Sn$.
By Lemma \ref{lem_RL}, we have an equality $s=r\circ l$ with $l\in\L$ and $r\in\R_\mathsf{sp}$.
Since $l$ is an $\mathfrak{s}_{\N}$-inflation and $\mathbb{E}_\N$ is closed, the composed morphism $g:=l\circ f'$ is still an $\mathfrak{s}_{\N}$-inflation.
The morphism $r$ is, by definition, a retraction, so there exists a corresponding section $r'$ which belongs to $\L_\mathsf{sp}$.
Taking a homotopy pushout of $f$ along $r'$ yields the following commutative diagram
\[
\xymatrix{
&\bullet\ar[r]^f\ar[d]_{r'}\ar@{}[dr]|{\rm (wPO)}&\bullet\ar[d]^{t'}\\
\bullet\ar[r]_g&\bullet\ar[r]_h&\bullet
}
\]
where $h$ is an $\mathfrak{s}_{\N}$-inflation and $t'\in\L_{\mathsf{sp}}$.
Again, we get a retraction $t\in\R_\mathsf{sp}$ corresponding to $t'$.
It is obvious that $\overline{t}\circ \overline{h}\circ \overline{g}=\overline{f}\circ \overline{r}\circ \overline{g}=\overline{f}\circ \overline{s}\circ \overline{f}'$ holds.
We have thus proved the assertion, since $h\circ g$ is an $\mathfrak{s}_{\N}$-inflation and $t\in\Sn$.
The remaining conditions can be checked dually.

Hence, due to Theorem \ref{Thm_Mult_Loc}(1), the pair $((\C,\mathbb{E}_\N,\mathfrak{s}_\N),\N)$ induces the extriangulated localization $(Q,\mu)\colon (\C,\mathbb{E}_\N,\mathfrak{s}_\N)\to (\widetilde{\C},\widetilde{\mathbb{E}}_\N,\widetilde{\mathfrak{s}}_\N)$ with respect to $\Sn$.
The identity $\widetilde{\C}\xto{\id}\widetilde{\C}_\N$ imposes an extriangulated structure on $\widetilde{\C}_\N$.

It remains to show the universality.
Let $(F,\phi)\colon (\C,\mathbb{E}_\N,\mathfrak{s}_\N)\to (\D,\mathbb{F},\mathfrak{t})$ be an exact functor with $\N\subseteq \Ker F$.
As we have a factorization $\Sn=\R_\mathsf{sp}\circ\L$, we have only to check that both $F(\R_\mathsf{sp})$ and $F(\L)$ consist of isomorphisms.
In fact, it directly follows from the exactness of $F$.
Thanks to Theorem \ref{Thm_Mult_Loc}(2), we have thus verified the universality.
\end{proof}

\emph{We denote by $(\widetilde{\C}_\N,\widetilde{\mathbb{E}}_\N,\widetilde{\mathfrak{s}}_\N)$ the extriangulated structure imposed through the identity $\widetilde{\C}\xto{\id}\widetilde{\C}_\N$ of additive categories}.
The exact functor $(Q,\mu)$ or the resulting category $(\widetilde{\C}_\N,\widetilde{\mathbb{E}}_\N,\widetilde{\mathfrak{s}}_\N)$ is called the \emph{extriangulated localization} of $\C$ with respect to $\N$.

We end this section by giving an example of such extriangulated localizations.
Let $(\U,\V)$ be a cotorsion pair in a triangulated category $(\C,\mathbb{E},\mathfrak{s})$.
We call $(\U,\V)$ a \emph{$t$-structure} of $\C$ if $\U[1]\subseteq \U$ (or equivalently, $\V\subseteq \V[1]$) is true.
(We remark that this is an equivalent condition so that $(\U,\V[2])$ forms a usual $t$-structure in the sense of \cite{BBD}, see also \cite[Ch. II]{Hap88}.)

\begin{example}\label{ex_relative_str_aisle}
Let $(\C^{\leq -1},\C^{\geq 1})$ be a $t$-structure of $\C$, namely, a cotorsion pair with $\C^{\leq -1}[1]\subseteq \C^{\leq -1}$.
Since $\N:=\C^{\geq 1}$ is extension-closed, there exists the extriangulated localization $(Q,\mu)\colon (\C,\mathbb{E}_\N,\mathfrak{s}_\N)\to(\widetilde{\C}_\N,\widetilde{\mathbb{E}}_\N,\widetilde{\mathfrak{s}}_\N)$ with respect to $\N$.
We put $\C^{\leq 0}:=\C^{\leq -1}[-1]$.
Then, the canonical inclusion $\mathsf{inc}: \C^{\leq 0}\hookrightarrow \C$ induces a natural exact equivalence $(F,\phi)\colon (\C^{\leq 0}, \mathbb{E}|_{\C^{\leq 0}},\mathfrak{s}|_{\C^{\leq 0}})\xto{\sim} (\widetilde{\C}_\N,\widetilde{\mathbb{E}}_\N,\widetilde{\mathfrak{s}}_\N)$.
\end{example}
\begin{proof}
To check the exactness of $F=Q\circ \mathsf{inc}$, we consider an $\mathfrak{s}$-triangle $A\overset{f}{\lra} B\overset{g}{\lra} C\overset{h}{\dra}$ in $(\C,\mathbb{E},\mathfrak{s})$ with all terms sitting in $\C^{\leq 0}$.
We have to show this $\mathfrak{s}$-triangle is indeed an $\mathfrak{s}_\N$-triangle in $(\C,\mathbb{E}_\N,\mathfrak{s}_\N)$.
Note that $A[1]$ also belongs to $\C^{\leq 0}$ by $\C^{\leq 0}[1]\subseteq \C^{\leq 0}$.
If there is a morphism $N\xto{x}C$ with $N\in\N$, then the composition $h\circ x$ obviously factors through $N\in\N\subseteq\N[1]$, which shows the condition ${\rm (Lex)}$.
Also, for any morphism $A\xto{x}N$ with $N\in\N$, we have $x\circ h[-1]=0$ by $(\C^{\leq 0},\N)=0$.
In particular, $x\circ h[-1]$ factors through $0\in\N[-1]$, showing ${\rm (Rex)}$.

It remains to show that $F$ is dense and fully faithful.
Note that any object $X\in\C$, by the definition of the cotorsion pair $(\C^{\leq -1},\N)$, admits a triangle
\[
X^{\geq 1}[-1]\xto{f} X^{\leq 0}\xto{g} X\xto{h}X^{\geq 1}
\]
with $X^{\leq 0}\in\C^{\leq 0}$ and $X^{\geq 1}\in\N$.
Since $X^{\geq 1}, X^{\geq 1}[-1]\in\N$, we get $g\in\Sn$ by Proposition \ref{prop_saturated}.
We have thus verified the density of $F$.
Let us consider objects $X,Y\in\C^{\leq 0}$.
To show the faithfulness, we assume that there exists a morphism $X\xto{a}Y$ satisfying $Q(a)=0$.
Corollary \ref{cor_phantom}(1) shows that $a$ factors through an object in $\N$.
Then $a=0$ follows from the $\Hom$-orthogonality $(\C^{\leq 0},\N)=0$.
Lastly, to show the fullness, let us consider a morphism $QX\xto{\alpha}QY$ in the quotient category $\widetilde{\C}_\N$.
Since $\overline{\Sn}$ is a multiplicative system, $\alpha$ is represented by a roof diagram  in $\overline{\C}$ such as $\overline{f}/\overline{s}\colon X\xto{\overline{a}}Y'\xleftarrow{\overline{s}}Y$ with $\overline{s}\in\overline{\Sn}$ (equivalently, $s\in\Sn$ by Proposition \ref{prop_saturated}).
Thus we have a fraction $f/s$ in $\C$ corresponding to $\alpha$.
Again, by Proposition \ref{prop_saturated}, we know the cone of $s$ factors through an object in $\N$, which shows that $a$ factors through $s$.
More precisely, there exists a morphism $X\xto{a'}Y$ with $a=s\circ a'$ in $\C$.
Hence, the equality $\alpha=(Qs)^{-1}Qa=Qa'$ guarantees the fullness of $F$.
\end{proof}

We should mention that a similar equivalence has been established in \cite[Thm. E]{Tat21} under a slight different circumstance\footnote{The author wishes to thank Kiriko Kato and Aran Tattar for useful informations concerning the example.}.

\section{Triangulated localizations}\label{sec_Triangulated}
Let $\C$ be a triangulated category and $\N$ a full subcategory of $\C$ which is closed under direct summands, isomorphisms and extensions.
In this section, we will prove that the resulting extriangulated category $\widetilde{\C}_\N$ corresponds to a triangulated category if and only if the given subcategory $\N$ is thick in the triangulated category $\C$.
In this case, the localization $(Q,\mu)\colon (\C,\mathbb{E}_\N,\mathfrak{s}_\N)\to (\widetilde{\C},\widetilde{\mathbb{E}}_\N,\widetilde{\mathfrak{s}}_\N)$ is the usual Verdier quotient.
Actually, we provide a more precise connection between our localization, the Verdier quotient and the biresolving quotient in Example~\ref{ex_list}(3).

Before stating the assertion, we recall the following basic fact on the Verdier quotient.

\begin{lemma}\label{lem_basic_Verdier}
Let us denote by $\mathsf{Ver}\colon \C\to\C/\N$ the Verdier quotient of a triangulated category $\C$ by a thick subcategory $\N\subseteq \C$.
Then, a morphism $s\in\Mor\C$ belongs to $\Sn$ if and only if $\mathsf{Ver}(s)$ is an isomorphism.
\end{lemma}
\begin{proof}
Let $A\overset{f}{\lra}B\overset{s}{\lra}C\overset{g}{\lra}A[1]$ be a triangle in $\C$.
If $\overline{f}=\overline{g}=0$,
applying $\mathsf{Ver}$ to the triangle, we get an isomorphism $\mathsf{Ver}(s)$.
Conversely, if $\mathsf{Ver}(s)$ is an isomorphism, we get $\mathsf{Ver}(f)=\mathsf{Ver}(g)=0$.
By the well-known property of $\mathsf{Ver}$ (e.g. \cite[Lem. 2.1.26]{Nee01}), we get $\overline{f}=\overline{g}=0$.
\end{proof}

Our localization the resulting category $\widetilde{\C}_\N$ of which is triangulated is characterized as follows.

\begin{theorem}\label{cor_tri}
We consider the extriangulated category $(\C,\mathbb{E}_\N,\mathfrak{s}_\N)$ and the localization $(Q,\mu)\colon (\C,\mathbb{E}_\N,\mathfrak{s}_\N)\to(\widetilde{\C}_\N,\widetilde{\mathbb{E}}_\N,\widetilde{\mathfrak{s}}_\N)$ with respect to the subcategory $\N$.
Then the following three conditions are equivalent.
\begin{enumerate}
\item[\textnormal{(i)}]
The extriangulated category $(\widetilde{\C}_\N,\widetilde{\mathbb{E}}_\N,\widetilde{\mathfrak{s}}_\N)$ corresponds to a triangulated category.
\item[\textnormal{(ii)}]
$\N$ is a thick subcategory of the triangulated category $(\C,\mathbb{E},\mathfrak{s})$.
\item[\textnormal{(iii)}]
$\N$ is a biresolving subcategory of the extriangulated category $(\C,\mathbb{E}_\N,\mathfrak{s}_\N)$.
\end{enumerate}
Under the above equivalent conditions, the localization $Q\colon (\C,\mathbb{E}_\N,\mathfrak{s}_\N)\to(\widetilde{\C}_\N,\widetilde{\mathbb{E}}_\N,\widetilde{\mathfrak{s}}_\N)$ coincides with the usual Verdier quotient.
\end{theorem}
\begin{proof}
(ii) $\Rightarrow$ (i):
Example \ref{ex_relative_str_thick} shows that $\mathbb{E}^L_\N(C,A)=\mathbb{E}^R_\N(C,A)=\mathbb{E}(C,A)$ holds for any $A,C\in\C$.
By Lemma \ref{lem_basic_Verdier}, the localization $Q:\C\to\widetilde{\C}_\N$ is nothing other than the Verdier quotient of $\C$ with respect to the thick subcategory $\N$.

(iii) $\Rightarrow$ (ii):
Let $N\in\N$ be given.
For the object $N[1]$, by definition, there exists an $\mathfrak{s}_\N$-triangle $N''\lra N'\lra N[1]\dashrightarrow$ with $N'\in\N$.
Since the $\mathfrak{s}_\N$-triangle is a part of a triangle $N\lra N''\lra N'\lra N[1]$ of $\C$ and $\N$ is extension-closed, we get $N''\in\N$.
It is sent to an $\widetilde{\mathfrak{s}}_\N$-triangle $QN''\lra QN'\lra QN[1]\dashrightarrow$, which gurantees $QN[1]\cong 0$ and $N[1]\in\N$.
Thus $\N$ is closed under suspensions $[1]$.
By the dual argument, we obtain $\N[-1]\subseteq\N$ and conclude that $\N$ is a thick subcategory in $(\C,\mathbb{E},\mathfrak{s})$.

(i) $\Rightarrow$ (iii):
Suppose that $\widetilde{\C}_\N$ is a triangulated category.
For an object $A\in\C$, we shall construct an $\mathfrak{s}_\N$-inflation $A\to N'$ with $N'\in\N$.
Let us consider a triangle $QA\overset{\alpha}{\lra} 0\lra QA[1]\overset{\id}{\lra} QA[1]$ in $\widetilde{\C}_\N$.
Since $QA\xto{\alpha} 0$ is an $\widetilde{\mathfrak{s}}_\N$-inflation,
by Lemma \ref{lem_inf},
$\alpha=\beta\circ Q(f)\circ \gamma$ holds for some isomorphisms $QA\xto{\gamma}QA', QN\xto{\beta}0$ and an $\mathfrak{s}_\N$-inflation $A'\xto{f}N$.
Note that Corollary \ref{cor_phantom} shows $N\in\N$ and $\gamma$ is represented by $A\xto{s}A''\xleftarrow{s'}A'$ with $s,s'\in\Sn$.
We may assume $s\in\L$.
In fact, by Lemma \ref{lem_RL}, any $s\in\Sn$ can be factorized as $s=s_2\circ s_1$ with $s_1\in\L, s_2\in\R_{\mathsf{sp}}$.
The retraction $s_2$ has a right inverse $s'_2\in\L_{\mathsf{sp}}$.
Thus, $\gamma$ is also represented by $A\xto{s_1}\bullet\xleftarrow{}A'$ with $s_1\in\L$.

Complete an $\mathfrak{s}_\N$-triangle $A'\overset{f}{\lra}N\overset{g}{\lra}C\overset{\delta}{\dra}$ and take a weak pushout of $f$ along $s'$ to get the following morphisms between $\mathfrak{s}_\N$-triangles.
\[
\xymatrix{
A'\ar[r]^f\ar[d]_{s'}\ar@{}[rd]|{\rm (wPO)}&N\ar[r]^g\ar[d]^{s''}&C\ar@{=}[d]\ar@{-->}[r]^\delta&\\
A''\ar[r]^{f'}&N'\ar[r]^{g'}&C\ar@{-->}[r]^{s'_*\delta}&
}
\]
By {\rm (MR3)}, we may assume $s''\in\Sn$.
Thus we get $N'\in\N$.
As $f'$ is an $\mathfrak{s}_\N$-inflation and $s\in\L$, the composed morphism $f'\circ s\colon A\to N'$ is a desired one.
Combining the dual argument, we finish the proof.
\end{proof}

\begin{remark}
Combining Example~\ref{ex_relative_str_thick} and Theorem~\ref{cor_tri},
if the extension-closed subcategory $\N\subseteq\C$ becomes biresolving in $(\C,\mathbb{E}_\N,\mathfrak{s}_\N)$, then we have $(\C,\mathbb{E}_\N,\mathfrak{s}_\N)=(\C,\mathbb{E},\mathfrak{s})$.
\end{remark}

\section{Exact localizations}\label{sec_Exact}
The aim of this section is to explore when the resulting category $\widetilde{\C}_\N$ is an exact category.
We keep our set-up on $\C$ and $\N$, that is, $\N$ is a full subcategory of a triangulated category $\C$ which is closed under direct summands, isomorphisms and extensions,
and consider the associated extriangulated category $(\C,\mathbb{E}_\N,\mathfrak{s}_\N)$ together with the localization $(Q,\mu)\colon (\C,\mathbb{E}_\N,\mathfrak{s}_\N)\to(\widetilde{\C}_\N,\widetilde{\mathbb{E}}_\N,\widetilde{\mathfrak{s}}_\N)$ with respect to $\N$.
\subsection{A necessary and sufficient condition}\label{subsec_Exact}

If $\N$ is functorially finite, Theorem \ref{cor_exact} provides an exact version of Theorem \ref{cor_tri}.
In advance, we prepare the following lemma for later use.

\begin{lemma}\label{lem_exact}
Assume that $\N$ satisfies the condition $\cone(\N,\N)=\C$ in the triangulated category $(\C,\mathbb{E},\mathfrak{s})$.
Then, the following equalities hold:
\begin{eqnarray*}
\mathbb{E}^L_\N(C,A)&=&\{h\in\mathbb{E}(C,A)\mid h[-1]\text{\ factors through an object\ }N\in\N\},\\
\mathbb{E}^R_\N(C,A)&=&\{h\in\mathbb{E}(C,A)\mid h\text{\ factors through an object\ }N\in\N\}.
\end{eqnarray*}
In particular, any $\mathfrak{s}_\N$-triangle $A\overset{f}{\lra} B\overset{g}{\lra} C\overset{h}{\dra}$ satisfies $\overline{h}=\overline{h[-1]}=0$.
Hence we have a monomorphism $Q(f)$ and an epimorphism $Q(g)$.
\end{lemma}
\begin{proof}
Let $A\overset{f}{\lra} B\overset{g}{\lra} C\overset{h}{\dra} $ be an $\mathfrak{s}_\N^L$-triangle of $\C$, that is, $C\xto{h}A[1]$ belongs to $\mathbb{E}_\N^L(C,A)$.
The assumption gives us a triangle $N'\lra N\overset{c}{\lra} C\lra N'[1]$ in $\C$ with $N,N'\in\N$.
We take a homotopy pullback of $g$ along $c$ and consider the commutative diagram below obtained by the octahedral axiom.
\[
\xymatrix{
&&N'\ar@{=}[r]\ar[d]&N'\ar[d]&\\
N[-1]\ar[d]_{c[-1]}\ar[r]^{h'[-1]}\ar[d]&A\ar[r]\ar@{=}[d]&B'\ar[r]\ar[d]&N\ar[d]^c\\
C[-1]\ar[r]^{h[-1]}&A\ar[r]^f&B\ar[r]^g&C\ar@{}[lu]|{\rm (wPO)}\\
}
\]
By the condition ${\rm (Lex)}$, $h'[-1]$ factors through an object in $\N$, namely, there exists a factorization $h'[-1]\colon N[-1]\xto{x_1} N''\xto{x_2} A$ with $N''\in\N$.
By taking homotopy pushout of $c[-1]$ along $x_1$, we get a  morphism of triangles:
\[
\xymatrix{
N[-1]\ar@{}[rd]|{\rm (wPO)}\ar[r]^{c[-1]}\ar[d]_{x_1}&C[-1]\ar[r]\ar[d]&N'\ar[r]\ar@{=}[d]&N\ar[d]\\
N''\ar[r]&M\ar[r]&N'\ar[r]&N''[1]
}
\]
Since $N', N''\in\N$ and $\N$ is extension-closed, it follows that $C[-1]\xto{h[-1]}A$ factors through $M\in\N$.
Hence, Proposition \ref{prop_saturated} shows that $Q(f)$ is monic.
A dual argument shows that $Q(g)$ is epic.
\end{proof}

The first aim of this subsection is to prove the assertion below.

\begin{theorem}\label{cor_exact}
Let us consider the following conditions.
\begin{enumerate}
\item[\textnormal{(i)}]
The extriangulated category $(\widetilde{\C}_\N,\widetilde{\mathbb{E}}_\N,\widetilde{\mathfrak{s}}_\N)$ corresponds to an exact category.
\item[\textnormal{(ii)}]
$\N$ satisfies the condition $\cone(\N,\N)=\C$ in the triangulated category $(\C,\mathbb{E},\mathfrak{s})$.
\item[\textnormal{(iii)}]
$\N$ is a Serre subcategory of the extriangulated category $(\C,\mathbb{E}_\N,\mathfrak{s}_\N)$.
\end{enumerate}
The condition \textnormal{(ii)} always implies \textnormal{(i)} and \textnormal{(iii)}.
Suppose that $\N$ is functorially finite in $\C$.
Then, the all conditions are equivalent.
\end{theorem}
\begin{proof}
(ii) $\Rightarrow$ (i):
Thanks to Lemmas \ref{lem_inf} and \ref{lem_exact}, we conclude that any $\widetilde{\mathfrak{s}}_\N$-inflation and $\widetilde{\mathfrak{s}}_\N$-deflation turn out to be monic and epic, respectively.
Therefore Proposition \ref{prop_extri_to_exact} guarantees that  the extriangulated structure on $\widetilde{\C}_\N$ corresponds to an exact category.

(ii) $\Rightarrow$ (iii): By the above argument, any $\mathfrak{s}_\N$-triangle $A\overset{f}{\lra} B\overset{g}{\lra} C\overset{h}{\dra} $ of $\C$ satisfies that $h$ and $h[-1]$ factor through objects in $\N$.
If $B\in\N$, by Lemma \ref{lem_extension_closed}, we get $A,C\in\N$.
Since $\N$ is extension-closed, it is a Serre subcategory of $(\C,\mathbb{E}_\N,\mathfrak{s}_\N)$.

\emph{In the rest of the proof, we suppose that $\N$ is functorially finite in $\C$.}
To confirm implications (i) $\Rightarrow$ (ii) and (iii) $\Rightarrow$ (ii), we take an object $C\in\C$, and consider a right $\N$-approximation $y$ of $C$ as a part of a triangle $C'\overset{x}{\lra} N\overset{y}{\lra} C\overset{z}{\lra} C'[1]$ with $N\in\N$.
Moreover, we consider a left $\N$-approximation $a\colon C'\to N'$ of $C'$ and take a homotopy pushout of $x$ along $a$, which induces a morphism of triangles.
\begin{equation}\label{diag_exact}
\xymatrix{
C'\ar[r]^x\ar[d]_a\ar@{}[rd]|{\rm (wPO)}&N\ar[r]^y\ar[d]^b&C\ar[r]^z\ar@{=}[d]&C'[1]\ar[d]\\
N'\ar[r]^{x'}&B\ar[r]^{y'}&C\ar[r]^{z'}&N'[1]
}
\end{equation}
Note that $\binom{x}{a}\colon C'\to N\oplus N'$ is a left $\N$-approximation of $C'$.
Moreover, $(b\ x')$ is also a right $\N$-approximation of $B$.
In fact, if there exists a morphism $f\colon N''\to B$ from an object $N''\in\N$, since $y$ is a right $\N$-approximation, $y'\circ f$ factors through $y$, say $y'\circ f=y\circ f'$ for some morphism $f'\colon N''\to N$.
By the equality $y'\circ (f-b\circ f')=(y-y'\circ b)\circ f'=0$, we conclude that $f$ factors through $N\oplus N'$.
Hence the triangle
\begin{equation}\label{eq_exact}
C'\overset{\binom{x}{-a}}{\lra} N\oplus N'\overset{(b\ x')}{\lra} B\lra C'[1]
\end{equation}
corresponding to ${\rm (wPO)}$ gives rise to an $\mathfrak{s}_\N$-triangle.

By using the $\mathfrak{s}_\N$-triangle (\ref{eq_exact}), we have the remaining implications.

(i) $\Rightarrow$ (ii):
Since any $\widetilde{\mathfrak{s}}_\N$-inflation is monic due to our assumption, the $\widetilde{\mathfrak{s}}_\N$-triangle $QC'{\lra} 0{\lra} QB\dra$ obtained from the triangle (\ref{eq_exact}) forces $C'\in\N$.
Therefore we obtain $C\in\cone(\N,\N)$.

(iii) $\Rightarrow$ (ii):
Since $\N$ is a Serre subcategory, the $\mathfrak{s}_\N$-triangle (\ref{eq_exact}) shows $C',B\in\N$.
Thus we conclude $C\in\cone(\N,\N)$.
\end{proof}

Next, we shall discuss on abelian exact structures on $\widetilde{\C}_\N$.

\begin{corollary}\label{cor_abelian}
Assume that $\cone(\N,\N)=\C$ holds in the triangulated category $(\C,\mathbb{E},\mathfrak{s})$.
Then, the following assertions hold.
\begin{enumerate}
\item[{\rm (1)}]
The resulting extriangulated category $(\widetilde{\C},\widetilde{\mathbb{E}}_\N,\widetilde{\mathfrak{s}}_\N)$ corresponds to an abelian exact category.
\item[{\rm (2)}]
The exact functor $(Q,\mu)\colon (\C,\mathbb{E}_\N,\mathfrak{s}_\N)\to(\widetilde{\C}_\N,\widetilde{\mathbb{E}}_\N,\widetilde{\mathfrak{s}}_\N)$ induces a cohomological functor $Q\colon (\C,\mathbb{E},\mathfrak{s})\to\widetilde{\C}_\N$ from the original triangulated category.
\item[{\rm (3)}] The cohomological functor $Q\colon (\C,\mathbb{E},\mathfrak{s})\to\widetilde{\C}_\N$ has the universality: Let $\A$ be an abelian category which gives rise to an extriangulated category.
For any cohomological functor $H\colon (\C,\mathbb{E},\mathfrak{s})\to \A$ with $\N\subseteq\Ker H$, there uniquely exists an exact functor $H'\colon \widetilde{\C}_\N\to\A$ with $H=H'\circ Q$.
\end{enumerate}
\end{corollary}

For the readability purpose, we divide the proof of Corollary \ref{cor_abelian} in several steps.

\begin{claim}\label{claim_abelian2}
Assume that $\cone(\N,\N)=\C$ holds.
Let us consider a triangle $A\overset{f}{\lra} B\overset{g}{\lra}N[1]\overset{h}{\lra}A[1]$ with $N\in\N$.
Then, $Q(f)$ is an $\widetilde{\mathfrak{s}}_\N$-inflation.
\end{claim}
\begin{proof}
By the assumption, as for the object $A$, there exists a triangle $A\overset{x'}{\lra} N'_1\overset{y'}{\lra} N'_2\overset{z'}{\lra} A[1]$ with $N'_1,N'_2\in\N$.
Applying the octahedral axiom to the composed morphism $f\circ z'[-1]$, we have the following commutative diagram made of triangles.
\[
\xymatrix{
&&N'_1\ar@{=}[r]\ar[d]&N'_1\ar[d]^{y'}\\
N'_2[-1]\ar[r]^{f'}\ar[d]_{z'[-1]}&B\ar[r]^{g'}\ar@{=}[d]&C\ar[d]^c\ar[r]^{h'}&N'_2\ar[d]^{z'}\\
A\ar[r]^f&B\ar[r]^g&N[1]\ar[d]_{-x'[1]\circ h}\ar[r]^h&A[1]\ar[d]^{-x'[1]}\ar@{}[lu]|{\rm (wPB)}\\
&&N'_1[1]\ar@{=}[r]&N'_1[1]
}
\]
Again we take a homotopy pullback of $g$ along $c$ to obtain the following commutative diagram made of triangles.
\begin{equation}\label{diag_abelian2}
\xymatrix{
&N'_1\ar@{=}[r]\ar[d]&N'_1\ar[d]&\\
A\ar[r]^{f''}\ar@{=}[d]&B'\ar[r]^{g''}\ar[d]^b&C\ar[d]^c\ar[r]^{h''}&A[1]\ar@{=}[d]\\
A\ar[r]^f&B\ar[r]^g\ar[d]&N[1]\ar@{}[lu]|{\rm (wPB)}\ar[d]^{-x'[1]\circ h}\ar[r]^h&A[1]\\
&N'_1[1]\ar@{=}[r]&N'_1[1]&
}
\end{equation}
The morphism $b$ is a retraction as $g$ factors through $c$, which guarantees $b\in\Sn$.
Hence we have an isomorphism $Q(f)\cong Q(f'')$.
Note that $f''$ is an $\mathfrak{s}_\N$-inflation.
In fact, $h''[-1]=(h\circ c)[-1]$ factors through $N\in\N$ and $h''=h\circ c=z'\circ h'$ factors through $N'_2\in\N$.
Therefore $Q(f)$ turns out to be an $\widetilde{\mathfrak{s}}_\N$-inflation.
\end{proof}

\begin{claim}\label{claim_abelian1}
If $\cone(\N,\N)=\C$ holds, then any morphism $\alpha\colon A\to B$ in $\widetilde{\C}_\N$ can be factorized as an epimorphism $\alpha_1$ followed by a monomorphism $\alpha_2$, namely, $\alpha=\alpha_2\circ\alpha_1$.
\end{claim}
\begin{proof}
Let $\alpha:A\to B$ be an morphism in $\widetilde{\C}_\N$.
Since $\overline{\Sn}$ is a multiplicative system in $\overline{\C}$,  it is represented by morphisms $A\xto{f}B'\xleftarrow{s}B$ with $s\in\Sn$, namely, $\alpha=Q(s)^{-1}\circ Q(f)$.
We embed the morphism $f$ into a triangle $A\overset{f}{\lra}B'\overset{g}{\lra}C\overset{h}{\lra}A[1]$ and, by $C\in\cone(\N,\N)$, get a triangle $N_2\overset{y}{\lra} N_1\overset{x}{\lra} C\lra N[1]$ with $N_1,N_2\in\N$.
Applying the octahedral axiom to the composed morphism $h\circ x$, we have the following commutative diagram
\begin{equation}\label{diag_abelian1}
\xymatrix{
&N_2\ar@{=}[r]\ar[d]&N_2\ar[d]^y&\\
A\ar[r]^{f'}\ar@{=}[d]&B''\ar[r]^{g'}\ar[d]^b&N_1\ar[d]^x\ar[r]^{h\circ x}&A[1]\ar@{=}[d]\\
A\ar[r]^f&B'\ar[r]^g&C\ar[r]^h\ar@{}[ul]|{\rm (wPB)}&A[1]
}
\end{equation}
Thus we get a desired factorization $f=b\circ f'$.
In fact, Proposition \ref{prop_saturated} shows that $Q(f')$ is epic and $Q(b)$ is monic in $\widetilde{\C}_\N$.
\end{proof}

Combining Claims \ref{claim_abelian2} and \ref{claim_abelian1}, we can prove Corollary \ref{cor_abelian}(1).

\begin{proof}[Proof of Corollary \ref{cor_abelian}(1)]
Theorem \ref{cor_exact} (ii) $\Rightarrow$ (i) shows that $\widetilde{\C}_\N$ corresponds to an exact category.
Hence it is enough to show that any morphism $\alpha\in\Mor\widetilde{\C}_\N$ is admissible in the sense that $\alpha=m\circ e$ for some $\widetilde{\mathfrak{s}}_\N$-deflation $e$ and $\widetilde{\mathfrak{s}}_\N$-inflation $m$.
Let $\alpha\colon A\to B$ be a morphism in $\widetilde{\C}_\N$.
By using the diagram (\ref{diag_abelian1}) in the proof of Claim \ref{claim_abelian1}, we may assume that there exists a factorization $\alpha=Q(b)\circ Q(f')$ with $\cone(f'),\cocone(b)\in\N$.
Due to Claim \ref{claim_abelian2} and its dual, $Q(f')$ and $Q(b)$ are $\widetilde{\mathfrak{s}}_\N$-deflation and $\widetilde{\mathfrak{s}}_\N$-inflation, respectively.
Therefore the exact structure on $\widetilde{\C}_\N$ is abelian.
\end{proof}

To show the assertion (2) in Corollary \ref{cor_abelian}, we need the following claim.

\begin{claim}\label{claim_cohomological1}
Assume that $\cone(\N,\N)=\C$ holds in the triangulated category $(\C,\mathbb{E},\mathfrak{s})$.
Let $N[-1]\overset{f}{\lra} B\overset{g}{\lra} C\overset{h}{\lra} N$ be a triangle in $\C$ with $N\in\N$.
Then, the induces sequence $QN[-1]\xto{Qf}QB\xto{Qg}QC\to 0$ forms an exact sequence in the abelian category $\widetilde{\C}_\N$.
\end{claim}
\begin{proof}
Let $N[-1]\overset{f}{\lra} B\overset{g}{\lra} C\overset{h}{\lra} N$ be a triangle with $N\in\N$.
By the condition $\cone(\N,\N)=\C$, the object $C$ admits a triangle $N_2\xto{y}N_1\xto{x}C\xto{g'}N_2[1]$ with $N_1,N_2\in\N$.
The octahedral axiom gives us the following commutative diagram made of triangles.
\begin{equation*}\label{diag_abelian3}
\xymatrix{
&N_2\ar@{=}[r]\ar[d]&N_2\ar[d]^y&\\
N[-1]\ar[r]^{f'}\ar@{=}[d]&A_0\ar[r]^{g'}\ar[d]^{f_0}&N_1\ar[d]^x\ar[r]^{}&N\ar@{=}[d]\\
N[-1]\ar[r]^f&B\ar[r]^g\ar[d]^{g_0}&C\ar[d]^{g'}\ar[r]^h\ar@{}[ul]|{\rm (wPB)}&N\\
&N_2[1]\ar@{=}[r]&N_2[1]&
}
\end{equation*}
Note that the square $({\rm wPB})$ corresponds to an $\mathfrak{s}_\N$-triangle $A_0\overset{\binom{-f_0}{g'}}{\lra}B\oplus N_1\overset{(g\, x)}{\lra}C\overset{}{\dashrightarrow}$.
In fact, the morphism $f'[1]\circ h\colon C\to A_0[1]$ obviously factors through $N\in\N$ and $\overline{f_0}\cong\overline{\binom{-f_0}{g'}}$ is monic in $\overline{\C}$.
Therefore we get a short exact sequence $0\to QA_0\xto{Qf_0}QB\xto{Qg}QC\to 0$ in $\widetilde{\C}_\N$.
Since $f=f_0\circ f'$ and $Qf'$ is epic, we have a desired exact sequence $QN[-1]\xto{Qf}QB\xto{Qg}QC\to 0$.
\end{proof}

Now we complete the proof of Corollary \ref{cor_abelian}.

\begin{proof}[Proof of Corollary \ref{cor_abelian}(2)(3)]
(2) Let $A\overset{f}{\lra} B\overset{g}{\lra} C\overset{h}{\lra} A[1]$ be a triangle in $(\C,\mathbb{E},\mathfrak{s})$.
We have to show that the induced sequence $QA\xto{Qf}QB\xto{Qg}QC$ is exact in $\widetilde{\C}_\N$, namely, $\Im Qf=\Ker Qg$.
First, using the condition $\cone(\N,\N)=\C$, we take a triangle $ N'\overset{y}{\lra} N\overset{x}{\lra} A[1]\overset{z}{\lra}N'[1]$ with $N,N'\in\N$.
Also, by the octahedral axiom, we have the following commutative diagram made of triangles.
\[
\xymatrix{
&&N'\ar[d]\ar@{=}[r]&N'\ar[d]^y&\\
N[-1]\ar[d]_{x[-1]}\ar[r]^{f'}&B\ar@{=}[d]\ar[r]^{g'}&C'\ar[d]_c\ar[r]^{h'}&N\ar[d]^x\\
A\ar[r]^f&B\ar[r]^g&C\ar[d]\ar[r]^h&A[1]\ar@{}[ul]|{\rm (wPB)}\ar[d]^z\\
&&N'[1\ar@{=}[r]]&N'[1]&
}
\]
Due to Proposition \ref{prop_saturated}, we have that $Q(x[-1])$ is epic and $Qc$ is monic.
Thus, we get equalities $\Im Qf=\Im Q(f\circ x[-1])=\Im Qf'$ and $\Ker Qg=\Ker Q(c\circ g')=\Ker Qg'$.
Since Claim \ref{claim_cohomological1} shows $\Im Qf'=\Ker Qg'$,
we also have $\Im Qf=\Ker Qg$ which guarantees the functor $Q$ is cohomological.

(3) A given cohomological functor $H\colon (\C,\mathbb{E},\mathfrak{s})\to \A$ restricts to an exact functor $(H,\phi)\colon (\C,\mathbb{E}_\N,\mathfrak{s}_\N)\to \A$.
Actually, thanks to Lemma \ref{lem_exact}, we know that any $\mathfrak{s}_\N$-triangle $A\overset{f}{\lra} B\overset{g}{\lra} C\overset{h}{\dra}$ is an $\mathfrak{s}$-triangle with $h,h[-1]\in [\N]$.
Thus, applying $H$ to this yields an (short) exact sequence $0\to HA\xto{Hf}HB\xto{Hg}HC\to 0$ in $\A$.
Note that $\A$ has a natural extriangulated structure corresponding to the abelian category.
So the cohomological functor $H$ gives rise to an exact functor of the form $(\C, \mathbb{E}_\N, \mathfrak{s}_\N) \to \A$, see \cite[Prop. 3.3]{Sak21} for details.
Similarly, we have an isomorphism $Hs\in\Mor\A$ for any $s\in\Sn$ by Proposition \ref{prop_description_of_Sn}.
By the universality of $(Q,\mu)$ in Theorem \ref{Thm_Mult_Loc}(2), there uniquely exists an exact functor $(H',\phi')\colon (\widetilde{\C}_\N,\widetilde{\mathbb{E}}_\N,\widetilde{\mathfrak{s}}_\N)\to \A$ with $(H,\phi)=(H',\phi')\circ (Q,\mu)$.
Since $\widetilde{\C}_\N$ is abelian, the additive functor $H\colon \widetilde{\C}_\N\to\A$ is exact in the usual sense.
\end{proof}

The following is straightforward by the proof of Corollary~\ref{cor_abelian}(2).

\begin{corollary}\label{cor_left_right_exact}
The exact functor $(Q,\mu)\colon (\C,\mathbb{E}_\N,\mathfrak{s}_\N)\to(\widetilde{\C}_\N,\widetilde{\mathbb{E}}_\N,\widetilde{\mathfrak{s}}_\N)$ induces a left exact functor $Q\colon (\C,\mathbb{E}^L_\N,\mathfrak{s}^L_\N)\to\widetilde{\C}_\N$ in the sense of \cite{Oga21}, namely, $Q$ sends any $\mathfrak{s}^L_\N$-triangle $A\overset{f}{\lra}B\overset{g}{\lra}C\dra$ to a left exact sequence $0\lra QA\overset{Qf}{\lra}QB\overset{Qg}{\lra}QC$ in $\widetilde{\C}_\N$.
Similarly, we have a right exact functor $Q\colon (\C,\mathbb{E}^R_\N,\mathfrak{s}^R_\N)\to\widetilde{\C}_\N$.
\end{corollary}

By the argument so far, we established half/left/right exact functors $Q\colon \C\to\widetilde{\C}_\N$ from an extension closed subcategory $\N$ with $\cone(\N,\N)=\C$ as depicted in the following commutative diagram,
\begin{equation*}\label{diag:one_sided_localization}
\xy
(0,0)*+{(\C,\mathbb{E},\mathfrak{s})}="0";
(20,14)*+{(\C,\mathbb{E}^R_\N,\mathfrak{s}^R_\N)}="2";
(20,-14)*+{(\C,\mathbb{E}^L_\N,\mathfrak{s}^L_\N)}="4";
(40,0)*+{(\C,\mathbb{E}_\N,\mathfrak{s}_\N)}="6";
(72,0)*+{\widetilde{\C}_\N}="8";
{\ar "2";"0"};
{\ar "4";"0"};
{\ar "6";"2"};
{\ar "6";"4"};
{\ar@/^24pt/^{\textnormal{right exact}} "2";"8"};
{\ar@/_24pt/_{\textnormal{left exact}} "4";"8"};
{\ar "6";"8"};
\endxy
\end{equation*}
where non-labeled arrows are exact functors.
The arrows in the left square are identities and the other arrows denote the functor $Q$ (as additive functors).

The following example shows that left/right exact functor in Corollary \ref{cor_left_right_exact} can be obtained from rigid subcategories, see Example \ref{ex_relative_str_rigid}.

\begin{example}\label{ex_relative_str_rigid2}
Let $\X$ be a rigid contravariantly finite subcategory of $\C$ and put $\N:=\{C\in\C\mid (\X,C)=0\}$.
Then we have an equality
\[
\mathbb{E}^R_\N(C,A)=\{h\in\mathbb{E}(C,A)\mid h\circ y=0\text{\ for all\ }y\colon X\to C\text{\ with\ }X\in\X\}.
\]
Moreover, the functor $(\X,-)\colon (\C,\mathbb{E}^R_\N,\mathfrak{s}^R_\N)\to\mod\X$ is right exact.
\end{example}
\begin{proof}
For any object $C\in\C$, a right $\X$-approximation $f\colon X\to C$ of $C$ yields a triangle $X\overset{f}{\lra}C\overset{g}{\lra}N\overset{h}{\lra}X[1]$ with $N\in\N$.
Since $X[1]\in\N$, we get $\cone(\N,\N)=\C$.
Thus, Lemma~\ref{lem_exact} tells us the equality is true.
The right-exactness of $(\X,-)$ follows from Corollary \ref{cor_left_right_exact}.
\end{proof}

\subsection{Ideal quotients by rigid subcategories}\label{subsec_rigid}
We shall discuss on the case that $\N$ is a rigid subcategory.
In this case, we shall show that the extriangulated localization $\widetilde{\C}_\N$ is equivalent to the ideal quotient $\overline{\C}:=\C/[\N]$.
The following characterizations are straightforward.

\begin{lemma}\label{lem_relative_str_rigid3}
If a given subcategory $\N$ is rigid, then we have the following equalities.
\begin{eqnarray*}
\mathbb{E}^L_\N(C,A) &=& \{h\in\mathbb{E}(C,A)\mid h\circ y=0\text{\ for all\ }y:N\to C\text{\ with\ }N\in\N\}\\
\mathbb{E}^R_\N(C,A) &=& \{h\in\mathbb{E}(C,A)\mid x\circ h[-1]=0\text{\ for all\ }x:A\to N\text{\ with\ }N\in\N\}
\end{eqnarray*}
\end{lemma}

Recall that the quotient functor $\C\to\widetilde{\C}_\N$ is factored as the ideal quotient $p:\C\to\overline{\C}$ followed by the localization $q\colon \overline{\C}\to\widetilde{\C}_\N$ with respect to $\overline{\Sn}$.
The rigidity of $\N$ makes the functor $q$ to be an equivalence.

\begin{corollary}\label{cor_localization_rigid}
If $\N$ is rigid, we have an equivalence $q\colon \overline{\C}\xto{\sim}\widetilde{\C}_\N$.
In particular, the ideal quotient $p\colon \C\to\overline{\C}$ gives rise to an exact functor $(p,\mu)\colon (\C,\mathbb{E}_\N,\mathfrak{s}_\N)\to(\overline{\C},\widetilde{\mathbb{E}}_\N,\widetilde{\mathfrak{s}}_\N)$.
\end{corollary}
\begin{proof}
It suffices to show that $\overline{\Sn}$ is the class of isomorphisms in $\overline{\C}$.
By definition, any morphism $f\in\L$ is embedded in an $\mathfrak{s}_\N$-triangle $A\overset{f}{\lra}B\overset{g}{\lra}N\overset{h}{\lra}A[1]$ with $N\in\N$.
The characterization in Lemma \ref{lem_relative_str_rigid3} forces $h=0$.
Thus, $f$ is a section in $\C$ and $\overline{f}$ is an isomorphism in $\overline{\C}$.
The dual argument holds for $\R$.
Since $\Sn$ is the class of finite compositions of morphisms in $\L\cup\R$, we have thus proved the assertion.
\end{proof}

\begin{remark}
Note that, if $\N$ is rigid in $(\C,\mathbb{E},\mathfrak{s})$, any object in $\N$ is projective and injective in $(\C,\mathbb{E}_\N,\mathfrak{s}_\N)$ in the sense that $\mathbb{E}_\N(\N,-)=\mathbb{E}_\N(-,\N)=0$ (see \cite[Def. 3.23]{NP19}).
Thus, due to \cite[Prop. 3.30]{NP19}, the ideal quotient $\overline{\C}$ has a natural extriangulated structure which coincides with $(\overline{\C},\widetilde{\mathbb{E}}_\N,\widetilde{\mathfrak{s}}_\N)$ in Corollary \ref{cor_localization_rigid}.
\end{remark}

Under the functorially finiteness of $\N$, we obtain the following characterization of $\N$ to make the ideal quotient $\overline{\C}=\C/[\N]$ abelian, which goes back to the well-known abelian structures in \cite{KR07, KZ08}.
We should remark that similar results were obtained as \cite[Thm. 27]{GJ15} and \cite[Cor. 1.4]{LZ19} in different contexts.
Recall that $\N$ is called a \emph{cluster-tilting} subcategory if it satisfies the rigidity and $\cone(\N,\N)=\C$.

\begin{corollary}\label{cor_abelian2}
Let $\C$ be a triangulated category and $\N$ a functorially finite rigid subcategory of $\C$.
Then, the ideal quotient $(\overline{\C},\widetilde{\mathbb{E}}_\N,\widetilde{\mathfrak{s}}_\N)$ is abelian if and only if $\N$ is a cluster-tilting subcategory.
\end{corollary}
\begin{proof}
If $\N$ is cluster-tilting, by definition, we get $\cone(\N,\N)=\C$.
Corollaries \ref{cor_abelian} and \ref{cor_localization_rigid} guarantee that $\overline{\C}$ is abelian.
Conversely, if $\overline{\C}$ is abelian, then Theorem \ref{cor_exact} forces $\cone({\N,\N})=\C$.
\end{proof}

\section{Relation to cohomological functors}\label{sec_Relation}
In this section, we shall study how our localization relates to cohomological functors.
Let $\C$ be a triangulated category.
In \cite{Sak21},
it was shown that any cohomological functor $H\colon \C\to \A$ naturally determines an extriangulated structure $(\C,\mathbb{E}_H,\mathfrak{s}_H)$ which makes $H\colon (\C,\mathbb{E}_H,\mathfrak{s}_H)\to\A$ exact.
Here $\A$ stands for an abelian category.
Of particular interest to us is the cases the structure $(\C,\mathbb{E}_H,\mathfrak{s}_H)$ coincides with our structure 
$(\C,\mathbb{E}_\N,\mathfrak{s}_\N)$ for $\N:=\Ker H$.
Such cases include the heart of $t$-structures in $\C$, the abelian quotient of $\C$ by a cluster-tilting subcategory as well as a more general concept, the heart of cotorsion pairs.
Our localization serves us a good understanding of such cohomological functors.
 
\subsection{Relative structures from cohomological functors}
Let us consider a cohomological (covariant) functor $H:\C\to\A$ from a triangulated category $\C$.
Remind that the extriangulated structure on $\C$ corresponding to the triangulated structure is denoted by $(\C,\mathbb{E},\mathfrak{s})$.
First, we recall the following result.

\begin{definition}\cite[Def. 3.2]{Sak21}
A subset $\mathbb{E}^R_H(C,A)$ of $\mathbb{E}(C,A)$ is defined to be a collection of $h$ which induces a triangle $A\overset{f}{\lra} B\overset{g}{\lra} C\overset{h}{\lra} A[1]$ with $H(g)$ epic in $\A$.
Similarly, we define a subset $\mathbb{E}^L_H(C,A)$ of $\mathbb{E}(C,A)$ consisting of $h$ such that $H(f)$ monic in $\A$.
Moreover, we put $\mathbb{E}_H(C,A):=\mathbb{E}^L_H(C,A)\cap\mathbb{E}^R_H(C,A)$.
\end{definition}

\begin{proposition}\cite[Prop. 3.3]{Sak21}
The following hold.
\begin{enumerate}
\item[\textnormal{(1)}]
Both $\mathbb{E}^R_H$ and $\mathbb{E}^L_H$ give rise to closed subfunctors of $\mathbb{E}$, hence we have an extriangulated structure $(\C,\mathbb{E}_H,\mathfrak{s}_H)$ relative to $(\C,\mathbb{E},\mathfrak{s})$, where $\mathfrak{s}_H$ is a restriction of $\mathfrak{s}$ to $\mathbb{E}_H$.
\item[\textnormal{(2)}]
$H$ gives rise to an exact functor $(H,\phi)\colon (\C,\mathbb{E}_H,\mathfrak{s}_H)\to\A$.
\end{enumerate}
\end{proposition}

Put $\N:=\Ker H$ and note that $\N$ is closed under direct summands, isomorphisms and extensions in the triangulated category $(\C,\mathbb{E},\mathfrak{s})$, where we can consider the extriangulated category $(\C,\mathbb{E}_\N,\mathfrak{s}_\N)$.
The following lemma follows from Theorem \ref{cor_exact} and Lemma \ref{lem_exact}.

\begin{lemma}\label{lem_relation_relative_str}
Suppose that $\cone(\N,\N)=\C$ and $\overline{f}=0$ if and only if $H(f)=0$ for any morphism $f\in\Mor\C$.
\begin{enumerate}
\item[\textnormal{(1)}]
The relative structure $(\C,\mathbb{E}_\N,\mathfrak{s}_\N)$ coincides with $(\C,\mathbb{E}_H,\mathfrak{s}_H)$.
\item[\textnormal{(2)}]
The cohomological functor $H$ uniquely factors through the localization $(Q,\mu)\colon (\C,\mathbb{E}_\N,\mathfrak{s}_\N)\to(\widetilde{\C}_\N,\widetilde{\mathbb{E}}_\N,\widetilde{\mathfrak{s}}_\N)$ as depicted below.
\[
\xymatrix{
(\C,\mathbb{E}_\N,\mathfrak{s}_\N)\ar[r]^{(Q,\mu)}\ar[d]_{(H,\phi)}&(\widetilde{\C}_\N,\widetilde{\mathbb{E}}_\N,\widetilde{\mathfrak{s}}_\N)\ar@{..>}[ld]^{(\widetilde{H},\widetilde{\phi})}\\
\A&
}
\]
Furthermore, the induced functor $(\widetilde{H},\widetilde{\phi})$ is a faithful exact functor.
\end{enumerate}
\end{lemma}
\begin{proof}
By the assumption and Lemma \ref{lem_exact}, we have the assertion (1).
As for (2), Theorem \ref{cor_exact} shows that there uniquely exists an exact functor $(\widetilde{H},\widetilde{\phi})\colon (\widetilde{\C}_\N,\widetilde{\mathbb{E}}_\N,\widetilde{\mathfrak{s}}_\N)\to\A$ such that $(H,\phi)=(\widetilde{H},\widetilde{\phi})\circ (Q, \mu)$.
To show the faithfulness of $\widetilde{H}$, let $\alpha\colon QA\to QB$ be a morphism in $\widetilde{\C}_\N$ which is represented by morphisms $A\xto{f}Y'\xleftarrow{s}Y$ with $s\in\Sn$ and assume $\widetilde{H}(\alpha)=0$.
Note that, due to (1), $H(s)$ is an isomorphism in $\A$.
Thus we get $\widetilde{H}(\alpha)=H(s)^{-1}\circ H(f)=0$.
By the assumption,  $H(f)=0$ implies $\overline{f}=0$.
Hence we have $Q(f)=0$ and $\alpha=0$.
\end{proof}

\subsection{Heart constructions}
We should recall from \cite{Nak11,AN12} the construction of the heart of cotorsion pair $(\U,\V)$ in a triangulated category $\C$.
Our aim of this subsection is to show that the heart can be regarded as the localization of $\C$ with respect to a certain subcategory $\N$.

A pair $(\U,\V)$ of full subcategories closed under direct summands in $\C$ is called a \emph{cotorsion pair} if the following conditions are fulfilled, which goes back to \cite{Sal79}.
\begin{itemize}
\item[-] $\mathbb{E}(\U,\V)=0$;
\item[-] $\cone(\V,\U)=\C=\cocone(\V,\U)$.
\end{itemize}

\begin{definition}
Let $\C$ be a triangulated category equipped with a cotorsion pair $(\U,\V)$ in $\C$, and put $\W:=\U\cap\V$.
We define the following associated subcategories:
$$
\C^-:=\U[-1]*\W;\ \C^+:=\W*\V[1];\ \H:=\C^+\cap\C^-.
$$
The ideal quotient $\H/[\W]$ is called the \textit{heart} of $(\U,\V)$.
\end{definition}

We should remark that the heart defined above goes back to that of $t$-structure in the sense of \cite{BBD}.
Let $(\C^{\leq -1},\C^{\geq 1})$ be a $t$-structure, namely, a cotorsion pair with $\C^{\leq -1}[1]\subseteq \C^{\leq -1}$.
The (usual) heart of $t$-structure is  defined to be $\C^{\leq -1}[-1]\cap \C^{\leq 1}[1]=\C^{\leq 0}\cap \C^{\leq 0}$.
Since $\W=\C^{\leq -1}\cap \C^{\geq 1}=0$, we have $\C^+=\C^{\leq 0}$ and $\C^-=\C^{\geq 0}$.
Hence, the heart $\H=\C^+\cap\C^-$ of the cotorsion pair coincides with the usual one of the $t$-structure.

We denote by $\pi\colon \C\to\C/[\W]$ the ideal quotient with respect to $\W$.
Abe and Nakaoka showed the following assertions.

\begin{lemma}\label{lem_coreflection}\cite[Def. 3.5, Rem. 3.6]{AN12}
For any $X\in\C$, there exists a commutative diagram
\begin{equation}\label{diag_coreflection}
\xymatrix@R=12pt@C=12pt{
V_X\ar[rr]&&X^-\ar[rr]^{\alpha_X}&&X\ar[rr]\ar[dr]&&V_X[1]\\
&&&&&V'_X\ar[ur]&
}
\end{equation}
where $V_X,V'_X\in\V$, $\alpha_X$ is a right $(\C^-)$-approximation of $X$ and the first row is a triangle.
Such a triangle is called a \emph{coreflection triangle} of $X$.
Dually, there exists a \emph{reflection triangle} of $X$:
\begin{equation}\label{diag_reflection}
\xymatrix@R=12pt@C=12pt{
U_X[-1]\ar[rr]\ar[rd]&&X\ar[rr]^{\beta_X}&&X^+\ar[rr]&&U_X\\
&U'_X\ar[ru]&&&&&
}
\end{equation}
where $U_X,U'_X\in\U$, $\beta_X$ is a left $(\C^+)$-approximation of $X$ and the first row is a triangle.
\end{lemma}
\begin{proof}
For later use, we recall a construction of a reflection triangle (\ref{diag_reflection}) of $X\in\C$.
Since $(\U,\V)$ is a cotorsion pair, we have a triangle $V'_X\lra U'_X\lra X\lra V'_X[1]$ with $U'_X\in\U, V'_X\in\V$.
Similarly, $U'_X\in\cone(\V[-1],\U[-1])$ shows the existence of a triangle $W_X[-1]\lra U_X[-1]\lra U'_X\lra W_X$ with $U_X\in\U, W_X\in\V$.
Note that the extension-closedness of $\U$ forces $W_X\in\W$.
Applying the octahedral axiom to the composable morphisms $U_X[-1]\to U'_X\to X$, we have the following commutative diagram made of triangles.
\begin{equation}\label{diag_const_coref}
\xymatrix{
&U_X[-1]\ar@{=}[r]\ar[d]&U_X[-1]\ar[d]&\\
V'_X\ar@{=}[d]\ar[r]&U'_X\ar[d]\ar[r]^{g'_X}\ar@{}[rd]|{\rm (wPO)}&X\ar[d]^{\beta_X}\ar[r]^{h'_X}&V'_X[1]\ar@{=}[d]\\
V'_X\ar[r]&W_X\ar[d]\ar[r]_{g_X}&X^+\ar[d]\ar[r]_{h_X}&V'_X[1]\\
&U_X\ar@{=}[r]&U_X&
}
\end{equation}
Then the third column is a reflection triangle.
 A construction of a coreflection triangle (\ref{diag_coreflection}) will be needed.
To construct a coreflection triangle, we use the dual of (\ref{diag_const_coref}) as follows.
\begin{equation}\label{diag_const_ref}
\xymatrix{
&V_X\ar@{=}[r]\ar[d]&V_X\ar[d]&\\
U'_X[-1]\ar@{=}[d]\ar[r]^{a_X}&X^-\ar[d]_{\alpha_X}\ar[r]\ar@{}[rd]|{\rm (wPB)}&W_X\ar[d]\ar[r]&U'_X\ar@{=}[d]\\
U'_X[-1]\ar[r]^{a'_X}&X\ar[d]\ar[r]&V'_X\ar[d]\ar[r]&U'_X\\
&V_X[1]\ar@{=}[r]&V_X[1]&
}
\end{equation}
The second column is a coreflection triangle of $X$.
\end{proof}

By a closer look at the above, it follows that the correspondence $X\mapsto X^+$ gives rise to a left adjoint $L$ of the inclusion $\C^+/[\W]\hookrightarrow\C/[\W]$.
Dually, a right adjoint $R$ of the inclusion $\C^-/[\W]\hookrightarrow\C/[\W]$ exists.
Moreover, the following holds.

\begin{lemma}\cite[Lem. 4.2]{AN12}\label{lem_AN}
The left adjoint $L\colon \C/[\W]\to\C^+/[\W]$ restricts to a functor $L \colon \C^-/[\W]\to\H/[\W]$.
Dually, the right adjoint $R\colon \C/[\W]\to\C^-/[\W]$ restricts to a functor $R \colon \C^+/[\W]\to\H/[\W]$.
Furthermore, there exists a natural isomorphism $\eta\colon LR\xto{\sim}RL$.
\end{lemma}

We turn our attention to the following cohomological functor.

\begin{theorem}\cite[Thm. 6.4]{Nak11}\cite[Thm. 5.7]{AN12}\label{thm:AN}
The heart $\H/[\W]$ is abelian.
Moreover, the functor $H:=LR\pi\colon \C\to\H/[\W]$ is cohomological.
\end{theorem}

In the rest, we shall show that the cohomological functor $H$ can be regarded as the localization of $\C$ with respect to $\N:=\Ker H$ the resulting category $\widetilde{\C}_\N$ of which corresponds to an abelian category.
The following sharpens the above theorem in the sense that it explains the universality of $H$ and how abelian exact structure of  the heart inherits from the given triangulated category $\C$.

\begin{theorem}\label{thm_heart_construction}
Let $H\colon \C\to\H/[\W]$ be the cohomological functor associated to a cotorsion pair $(\U,\V)$ in $\C$.
Then the following assertions hold.
\begin{enumerate}
\item[\textnormal{(1)}]
The functor $H$ gives rise to an exact functor $(H,\phi)\colon (\C,\mathbb{E}_\N,\mathfrak{s}_\N)\to\H/[\W]$.
\item[\textnormal{(2)}]
The localization $(\widetilde{\C}_\N,\widetilde{\mathbb{E}}_\N,\widetilde{\mathfrak{s}}_\N)$ of $\C$ with respect to $\N$ corresponds to an abelian category.
\item[\textnormal{(3)}]
The induced unique functor $(\widetilde{H},\widetilde{\phi})\colon (\widetilde{\C}_\N,\widetilde{\mathbb{E}}_\N,\widetilde{\mathfrak{s}}_\N)\to\H/[\W]$ is an exact equivalence.
\end{enumerate}
\end{theorem}

To prove the theorem, we reveal some properties of $\N$.
Let us begin with a bit more observations on (co)reflection triangles.

\begin{lemma}\label{lem_coreflection2}
Let $f\colon A\to B$ be a morphism in $\C$ and consider reflection triangles (\ref{diag_reflection}) of $A$ and $B$.
If the induced morphism $f^+\colon A^+\to B^+$ factors through an object in $\W$, then $f$ factors through an object in $\U$.
\end{lemma}
\begin{proof}
We consider commutative diagrams (\ref{diag_const_coref}) for $X=A,B$.
Since $\beta_X$ is a left $(\C^+)$-approximation of $X$, we get a commutative diagram below.
\[
\xymatrix{
A\ar[r]^{\beta_A}\ar[d]_{f}&A^+\ar[d]^{f^+}\\
B\ar[r]^{\beta_B}&B^+
}
\]
By the assumption, $f^+$ factors through $W\in\W$ and $h_B\circ f^+=0$.
Thus we have $h'_B\circ f=h_B\circ \beta_B\circ f=h_B\circ f^+\circ \beta_A=0$.
Hence, the second row of (\ref{diag_const_coref}) shows that $f$ factors through $U'_B\in\U$.
\end{proof}

The kernel of $H$ has the following nice properties which permit us to use Lemma \ref{lem_relation_relative_str}.

\begin{proposition}\label{prop_kernel_heart}
The following hold for the kernel $\N:=\Ker H$.
\begin{enumerate}
\item[\textnormal{(1)}]
$\N=\add(\U*\V)$ is true. In particular, $\cone(\N,\N)=\C$ holds.
\item[\textnormal{(2)}]
Let $f\in\Mor\C$ be given.
Then $H(f)=0$ if and only if $f$ factors through an object in $\N$.
\end{enumerate}
\end{proposition}
\begin{proof}
The assertion (1) is proved in \cite[Cor. 3.8]{LN19}.
To prove the assertion (2), let $f\colon A\to B$ be a morphism in $\C$.
Since the `if' part is obvious, to prove the `only if' part, we assume $H(f)=0$.
By taking coreflection triangles of $A, B$ and reflection triangles of $A^-, B^-$ successively, we get the following commutative squares $(*)$ and $(**)$:
\[
\xymatrix{
A\ar@{}[rd]|{(*)}\ar[d]_f&A^-\ar@{}[rd]|{(**)}\ar[d]_{f^-}\ar[r]^{\beta_{A^-}}\ar[l]_{\alpha_A}&A^\pm\ar[d]^{f^\pm}\\
B&B^-\ar[r]^{\beta_{B^-}}\ar[l]_{\alpha_B}&B^\pm
}
\]
where we abbreviate the symbol $({X^-})^+$ by $X^\pm$.
By definition of $H$, $f^\pm$ factors through an object in $\W$.
Thanks to Lemma \ref{lem_coreflection2}, $f^-$ factors through an object $U\in\U$.
We consider the diagrams (\ref{diag_const_ref}) for $X=A,B$ and notice that $f\circ a'_A=f\circ \alpha_A \circ a_A=\alpha_B\circ f^-\circ a_A$ is also factors through $U$, say $f\circ a'_A\colon U'_A[-1]\xto{b}U\xto{c}B$.
Taking a homotopy pushout of $a'_A$ along $b$ yields the following morphism of triangles.
\[
\xymatrix{
U'_A[-1]\ar[r]^{a'_A}\ar[d]_{b}\ar@{}[rd]|{\rm (wPO)}&A\ar[r]\ar[d]&V'_A\ar[r]\ar@{=}[d]&U'_A\ar[d]\\
U\ar[r]&K\ar[r]&V'_A\ar[r]&U[1]
}
\]
The commutative square corresponding to $c\circ b=f\circ a'_A$ shows that $f$ factors through $K\in\U*\V$.
\end{proof}

The following description of morphisms in the heart $\H/[\W]$ is fundamental, e.g., \cite[Cor. 3.11]{Oga22}.

\begin{lemma}\label{lem_description_morph}
Consider a morphism $\gamma\colon H(X)\to H(Y)$ in $\H/[\W]$.
We also let $\alpha_X\colon X^-\to X$ and $\beta_Y\colon Y\to Y^+$ be a right $\C^-$-approximation of $X$ and a left $\C^+$-approximation of $Y$ in (\ref{diag_coreflection}) and (\ref{diag_reflection}), respectively.
Then, there exists a morphism $c\colon X^-\to Y^+$ in $\C$ such that $\gamma=H(\beta_Y)^{-1}\circ H(c)\circ H(\alpha_X)^{-1}$.
\end{lemma}

Now we are in position to prove Theorem \ref{thm_heart_construction}.

\begin{proof}[Proof of Theorem \ref{thm_heart_construction}]
Due to Proposition \ref{prop_kernel_heart}, Lemma \ref{lem_relation_relative_str} shows that the relative extriangulated structure $(\C,\mathbb{E}_\N,\mathfrak{s}_\N)$ with respect to $\N=\Ker H$ coincides with $(\C,\mathbb{E}_H,\mathfrak{s}_H)$ determined by $H$.
Since $\cone(\N,\N)=\C$ is true, the extriangulated category $(\widetilde{\C}_\N,\widetilde{\mathbb{E}}_\N,\widetilde{\mathfrak{s}}_\N)$ corresponds to an abelian category by Corollary \ref{cor_abelian}(1).
Thus, by Lemma \ref{lem_relation_relative_str}, the unique exact functor $(\widetilde{H},\widetilde{\phi})\colon (\widetilde{\C}_\N,\widetilde{\mathbb{E}}_\N,\widetilde{\mathfrak{s}}_\N)\to\H/[\W]$ is faithful.
The functor $H\colon \C\to\H/[\W]$ is essentially surjective, so is $\widetilde{H}$.
To show the fullness of $\widetilde{H}$, we consider a morphism $\gamma\colon HA\to HB$ in $\H/[\W]$.
By Lemma \ref{lem_description_morph}, considering a coreflection triangle (\ref{diag_coreflection}) of $A$ and a reflection triangle (\ref{diag_reflection}) of $B$, we have $\gamma=H(\beta_B)^{-1}\circ H(f) \circ H(\alpha_A)^{-1}$ for some morphism $f\colon A^-\to B^+$.
Since $\alpha_A, \beta_B\in\Sn$, the morphism $Q(\beta_B)^{-1}\circ Q(f) \circ Q(\alpha_A)^{-1}$ exists and corresponds to $\gamma$ through $\widetilde{H}$.
We conclude that $\widetilde{H}$ is an equivalence.
Since the structure $(\widetilde{\C}_\N,\widetilde{\mathbb{E}}_\N,\widetilde{\mathfrak{s}}_\N)$ corresponds to an abelian exact structure and $\H/[\W]$ is abelian, we have thus obtained a desired exact equivalence $(\widetilde{H},\widetilde{\phi})$.
\end{proof}

\subsection{Examples}
Let us demonstrate how our localization relates to some known constructions of cohomological functors.
Throughout this subsection, $\C=(\C,\mathbb{E},\mathfrak{s})$ is a triangulated category.

\subsubsection{Cluster tilting subcategory}
Let us consider a cluster tilting subcategory $\X\subseteq\C$, namely, the pair $(\X,\X)$ forms a cotorsion pair.
The following well-known results are due to \cite{KZ08, KR07}:
\begin{enumerate}
\item[{\rm (a)}]
The ideal quotient $\C/[\X]$ is abelian; and
\item[{\rm (b)}]
The natural functor $H\colon \C\to\C/[\X]$ is cohomological.
\end{enumerate}
These phenomena can be understood through our extriangulated localization as follows.
Since $\X$ is extension-closed, putting $\N= \X$, we have the extriangulated localization $(Q,\mu)\colon (\C,\mathbb{E}_\N,\mathfrak{s}_\N)\to (\widetilde{\C}_\N,\widetilde{\mathbb{E}}_\N,\widetilde{\mathfrak{s}}_\N)$ by Theorem \ref{thm_main1}.
Also, since $\N$ is rigid, we have a natural equivalence $\C/[\N]\xto{\sim}\widetilde{\C}_\N$ and $Q$ is just the ideal quotient by Corollary \ref{cor_localization_rigid}.
Lastly, since $\cone(\N,\N)=\C$ holds in the triangulated category $\C$, Corollary \ref{cor_abelian}(1) shows that $(\widetilde{\C}_\N,\widetilde{\mathbb{E}}_\N,\widetilde{\mathfrak{s}}_\N)$ corresponds to an abelian category, which explains the first item (a).
Again, Corollary \ref{cor_abelian}(2) says that the functor $Q$ induces a cohomological functor $Q\colon (\C,\mathbb{E},\mathfrak{s})\to \widetilde{\C}_\N$ which coincides with the one in the second item (b).

\subsubsection{Rigid subcategory}
As a generalized setup of the above, we consider a contravariantly finite and rigid subcategory $\X\subseteq\C$ which is closed under taking direct summands.
Let $k$ be a field.
We moreover assume that $\C$ is a $k$-linear $\Hom$-finite Krull-Schmidt triangulated category with a Serre functor.
Put $\N=\Ker(\X,-)$.
Buan-Marsh investigated a natural cohomological functor $(\X,-)\colon \C\to\mod\X$ in terms of Gabriel-Zisman localization passing to the ideal quotient $p\colon \C\to\C/[\N]=\overline{\C}$.
In fact, they noticed that the functor $(\X,-)$ is factored as follows,
\begin{equation}\label{Buan-Marsh's_localization}
\xymatrix{
\C\ar[rd]_{p}\ar[rr]^{(\X,-)}&&\mod\X\\
&\overline{\C}\ar[ru]_{\mathsf{Loc}}&
}
\end{equation}
where $\mathsf{Loc}$ is induced by the universality of $\overline{\C}$.
We focus on the following assertion which is a part of \cite[Thm. 5.1]{BM13b} (see also \cite[Thm. 4.6]{Bel13}).

\begin{proposition}\label{prop_BM}
Put $\mathscr{S}$ to be the class of morphisms $s$ in $\C$ with $(\X,s)$ being an isomorphism in $\mod\X$.
Then the natural functor $\mathsf{Loc}$ is realized as a Gabriel-Zisman localization for the multiplicative system $\overline{\mathscr{S}}\subseteq\Mor\overline{\C}$.
\end{proposition}

We shall prove Proposition \ref{prop_BM} in the context of our extriangulated localization to clarify their connection.
Note that $\cone(\N,\N)=\C$ holds in $(\C,\mathbb{E},\mathfrak{s})$.
Thus, similarly to the cluster tilting case,
the extension-closed subcategory $\N$ induces the extriangulated localization $(Q,\mu)\colon (\C,\mathbb{E}_\N,\mathfrak{s}_\N)\to (\widetilde{\C}_\N,\widetilde{\mathbb{E}}_\N,\widetilde{\mathfrak{s}}_\N)$ with $\widetilde{\C}_\N$ abelian.
Remind our construction of $\widetilde{\C}_\N$: it is defined as a certain Gabriel-Zisman localization by passing to the ideal quotient $\overline{\C}$.
More precisely, it can be described as follows,
\begin{equation}
\xymatrix{
\C\ar[rd]_{p}\ar[rr]^{Q\quad\quad}&&\widetilde{\C}_\N=\overline{\C}[\overline{\Sn}^{-1}]\\
&\overline{\C}\ar[ru]_{\overline{Q}}&
}
\end{equation}
where $\overline{Q}$ denotes the localization at $\overline{\Sn}$.
It suffices to show that $\widetilde{\C}_\N$ is equivalent to $\mod\X$ in a natural way.
Actually, if this is true, we have $\mathscr{S}=\Sn$ by the saturatedness of $\Sn$, see Proposition \ref{prop_saturated}.
Also, we know $\overline{\Sn}$ is a multiplicative system by (MR2) in Theorem \ref{Thm_Mult_Loc}.

\begin{claim}
There exists an exact equivalence $F\colon \widetilde{\C}_\N\to\mod\X$ with $(\X,-)=F\circ Q$.
\end{claim}
\begin{proof}
Recall from Corollary \ref{cor_abelian} that $Q$ induces a cohomological functor $Q\colon \C\to\widetilde{\C}_\N$ satisfying a universality.
As $\N=\Ker(\X,-)$, the universality of $Q$ guarantees the existence of an exact functor $F\colon \widetilde{\C}_\N\to\mod\X$ with $(\X,-)=F\circ Q$.
Lastly, we shall show that $F$ is dense and fully faithful.

To show the density of $F$, let $M$ be an object in $\mod\X$ with a projective presentation $\X(-,X_1)\xto{f\circ -}\X(-,X_0)\to M\to 0$.
The morphism $f$ is completed into a triangle $X_1\overset{f}{\lra}X_0\overset{g}{\lra}C(f)\overset{h}{\lra}X_1[1]$ in $\C$.
Since $\X$ is rigid, we have an isomorphism $M\cong (\X,C(f))$ in $\mod\X$.

To show the faithfulness, we consider a morphism $QX_1\xto{\alpha}QX_0$ in $\widetilde{\C}_\N$ with $F(\alpha)=0$.
Remind that $\alpha$ is represented by a roof diagram $X_1\xto{f}X'_0\xleftarrow{s}X_0$ in $\C$ with $s\in\Sn$.
Both the cone and the cocone of $s$ factors through an object in $\N$, so $(\X,s)$ is an isomorphism.
Thus, we have $F(\alpha)=(\X,s)^{-1}(\X,f)=0$ and $(\X,f)=0$.
Since $\X$ is contravariantly finite and rigid, $f$ factors through an object in $\N$, which shows $\alpha=(Qs)^{-1}Qf=0$.

It remains to check the fullness. Fix objects $C, C'$ in $\C$ and consider $M=(\X,C), M'=(\X,C')$ in $\mod\X$.
We also consider a morphism $\alpha\colon M\to M'$.
It is obvious that $M$ (resp. $M'$) has a projective presentation.
However, we give an explicit construction for later use.
By considering a right $\X$-approximation of $C$, we have a triangle $N_0[-1]\overset{a_0}{\lra} X_0\overset{b_0}{\lra} C\lra N_0$ with $X_0\in\X$ and $N_0\in\N$.
Also, for the object $N_0[-1]$, we have a triangle $N_1[-1]\overset{a_1}{\lra} X_1\overset{b_1}{\lra} N_0[-1]\lra N_1$ with $X_1\in\X$ and $N_1\in\N$ in the same manner.
Applying the octahedral axiom to $a_0b_1$, we have the following commutative diagram of triangles.
\begin{equation}
\xymatrix{
&&N_1\ar@{=}[r]\ar[d]&N_1\ar[d]^{a_1[1]}\\
X_1\ar[r]^{a_0b_1}\ar[d]_{b_1}&X_0\ar[r]\ar@{=}[d]&D\ar[r]\ar[d]^s&X_1[1]\ar[d]\\
N_0[-1]\ar[r]^{a_0}&X_0\ar[r]^{b_0}&C\ar[r]\ar[d]&N_0\ar[d]\\
&&N_1[1]\ar@{=}[r]&N_1[1]
}
\end{equation}
Note that the third column yields an isomorphism $(\X,s)\colon (\X,D)\xto{\sim}(\X,C)$ and $s\in\Sn$.
Hence the second row gives rise to a projective presentation $(\X,X_1)\to (\X,X_0)\to M\to 0$.
Similarly, we have a projective presentation $(\X,X'_1)\to (\X,X'_0)\to M'\to 0$ for $M'$ via a morphism $D'\xto{s'}C'$ in $\Sn$.
Since the given morphism $M\xto{\alpha}M'$ extends to a morphism of their projective presentations,
via the Yoneda functor, we have a morphism of triangles as below.
\[
\xymatrix{
X_1\ar[r]^{a_0b_1}\ar[d]&X_0\ar[r]\ar[d]&D\ar[r]\ar[d]^f&X_1[1]\ar[d]\\
X'_1\ar[r]^{a'_0b'_1}&X'_0\ar[r]&D'\ar[r]&X'_1[1]
}
\]
The morphisms $C\xleftarrow{s}D\xto{f}D'\xto{s'}C'$ obtained above gives rise to a desired morphism $(Qs')(Qf)(Qs)^{-1}$ in $\widetilde{\C}_\N$.
In fact, we have $\alpha=(\X,s')(\X,f)(\X,s)^{-1}$.
\end{proof}

\medskip
\noindent
{\bf Acknowledgements.}
The author is grateful for Hiroyuki Nakaoka, Arashi Sakai an Shunya Saito for their interest and valuable comments.
He also would like to show his gratitude to the anonymous referee who found and corrected a mistake in the preliminary version.
This work is supported by JSPS KAKENHI Grant Number JP22K13893.

%
%
%
%



\begin{thebibliography}{99}

\bibitem[AN]{AN12}
N.~Abe, H.~Nakaoka, \emph{General heart construction on a triangulated category (II): Associated homological functor}.
Appl. Categ. Structures \textbf{20} (2012), no. 2, 161--174.


\bibitem[AS]{AS93}
M.~Auslander, {\O}.~Solberg,
\emph{Relative homology and representation theory. I. Relative homology and homologically finite subcategories}.
Comm. Algebra \textbf{21} (1993), no. 9, 2995--3031.

\bibitem[BBD]{BBD}
A.~Beilinson, J.~Bernstein, P.~Deligne, \emph{Faisceaux Pervers (Perverse sheaves)}.
Analysis and Topology on Singular Spaces, I, Luminy, 1981, Asterisque 100 (1982) 5--171 (in French).

\bibitem[Bel1]{Bel00}
A.~Beligiannis,
\emph{Relative homological algebra and purity in triangulated categories}.
J. Algebra \textbf{227} (2000), no. 1, 268--361.

\bibitem[Bel2]{Bel13}
A.~Beligiannis,
\emph{Rigid objects, triangulated subfactors and abelian localizations}.
Math. Z. \textbf{274} (2013), no. 3-4, 841--883.

\bibitem[BTHSS]{BTHSS23}
R.~Bennett-Tennenhaus, J.~Haugland, M.~H. Sand{\o}y, A.~Shah,
\emph{The category of extensions and a characterisation of {$n$}-exangulated functors}.
Math. Z., \textbf{305}(3):44, 2023.

\bibitem[B-TS]{B-TS21} R.~Bennett-Tennenhaus, A.~Shah, \emph{Transport of structure in higher homological algebra}.
J. Algebra \textbf{574} (2021), 514--549.

\bibitem[BM1]{BM13a}
A.~B.~Buan, B.~R.~Marsh, \emph{From triangulated categories to module categories via localisation}.
Trans. Amer. Math. Soc. \textbf{365} (2013), no. 6, 2845--2861.

\bibitem[BM2]{BM13b}
A.~B.~Buan, B.~R.~Marsh, \emph{From triangulated categories to module categories via localization II: calculus of fractions}.
J. Lond. Math. Soc. (2) \textbf{86} (2012), no. 1, 152--170.

\bibitem[C-E]{C-E98} M.~E.~C\'{a}rdenas-Escudero, \emph{Localization for exact categories}.
\newblock ProQuest LLC, Ann Arbor, MI, 1998.
\newblock Thesis (Ph.D.)--State University of New York at Binghamton.

\bibitem[DRSSK]{DRSSK99}
P.~Dr\"{a}xler, I.~Reiten, S.~Smal{\o}, {\O}.~Solberg,
\emph{Exact categories and vector space categories}.
With an appendix by B. Keller. Trans. Amer. Math. Soc. \textbf{351} (1999), no. 2, 647--682.

\bibitem[Eno]{Eno21}
H.~Enomoto,
\emph{Classifying substructures of extriangulated categories via Serre subcategories}.
Appl. Categ.
Structures, \textbf{29}(6):1005-1018, 2021.

\bibitem[Fri]{Fri08}
T.~Fritz, \emph{Categories of Fractions Revisited}.
Morfismos 15 (2011), no. 2, 19-38.

\bibitem[Gab]{Gab62}
P.~Gabriel, \emph{Des cat\'egories ab\'eliennes}.
Bull. Soc. Math. France \textbf{90} (1962), 323--448.

\bibitem[GZ]{GZ67}
P.~Gabriel, M.~Zisman, \emph{Calculus of fractions and homotopy theory}.
Ergebnisse der Mathematik und ihrer Grenzgebiete, Band 35 Springer-Verlag New York, Inc., New York 1967 {\rm x}+168 pp.

\bibitem[GJ]{GJ15}
B.~Grimeland, K.~Jacobsen, \emph{Abelian quotients of triangulated categories}.
J. Algebra \textbf{439} (2015), 110--133.

\bibitem[Hap]{Hap88}
D.~Happel, \emph{Triangulated categories in the representation theory of finite dimensional algebras}.
London Mathematical Society Lecture Note Series, 119. Cambridge University Press, Cambridge, 1988.

\bibitem[Hau]{Hau21}
J.~Haugland,
\emph{The {G}rothendieck group of an {\(n\)}-exangulated category}.
Appl. Categ. Structures, \textbf{29}(3):431--446, 2021.

\bibitem[HLN]{HLN21}
M.~Herschend, Y.~Liu, H.~Nakaoka,
\emph{$n$-exangulated categories (I): Definitions and fundamental properties}.
J. Algebra \textbf{570} (2021), 531—586.

\bibitem[HS]{HS20}
S.~Hassoun, A.~Shah, \emph{Integral and quasi-abelian hearts of twin cotorsion pairs on extriangulated categories}.
Comm. Algebra \textbf{48} (2020), no. 12, 5142--5162.

\bibitem[INP]{INP19}
O.~Iyama, H.~Nakaoka, Y.~Palu, \emph{Auslander--Reiten theory in extriangulated categories}. Trans. Amer. Math. Soc. Ser. B 11 (2024), 248--305.

\bibitem[JP]{JP22}
P.~J{\o}rgensen, Y.~Palu, \emph{A Caldero-Chapoton map for infinite clusters}.
Trans. Amer. Math. Soc., \textbf{365}(3):1125--1147, 2013.

\bibitem[JS]{JS22}
P.~J{\o}rgensen, A.~Shah, \emph{The Index With Respect to a Rigid Subcategory of a Triangulated Category}.
Int. Math. Res. Not. IMRN {\bf 2024}, no. 4, 3278--3309.

\bibitem[KR]{KR07}
B.~Keller, I.~Reiten, \emph{Cluster-tilted algebras are Gorenstein and stably Calabi-Yau}.
Adv. Math. \textbf{211} (2007), no. 1, 123--151.

\bibitem[Kla]{Kla22}
C.~Klapproth, \emph{$n$-extension closed subcategories of $n$-exangulated categories}.
arXiv:2209.01128v3

\bibitem[KZ]{KZ08}
S.~Koenig, B.~Zhu, \emph{From triangulated categories to abelian categories: cluster tilting in a general framework}.
Math. Z. \textbf{258} (2008), no. 1, 143--160.

\bibitem[Kra]{Kra00}
H.~Krause,
\emph{Smashing subcategories and the telescope conjecture--an algebraic approach}.
Invent. Math. \textbf{139} (2000), no. 1, 99--133.

\bibitem[Liu]{Liu17}
Y.~Liu, \emph{Localizations of the hearts of cotorsion pairs associated with mutations}.
Glasg. Math. J., \textbf{62}(3), 564-583. 

\bibitem[LN]{LN19}
Y.~Liu, H.~Nakaoka, \emph{Hearts of twin cotorsion pairs on extriangulated categories}.
J. Algebra \textbf{528} (2019), 96--149.

\bibitem[LZ]{LZ19}
Y.~Liu, P.~Zhou, \emph{Abelian categories arising from cluster tilting subcategories II: quotient functors}.
Proc. Roy. Soc. Edinburgh Sect. A \textbf{150} (2020), no. 6, 2721--2756. 

\bibitem[Nak1]{Nak11}
H.~Nakaoka, \emph{General heart construction on a triangulated category (I): Unifying $t$-structures and cluster tilting subcategories}.
Appl. Categ. Structures \textbf{19} (2011), no. 6, 879--899.

\bibitem[Nak2]{Nak13}
H.~Nakaoka, \emph{General heart construction for twin torsion pairs on triangulated categories}.
J. Algebra \textbf{374} (2013), 195--215.

\bibitem[NOS]{NOS21}
H.~Nakaoka, Y.~Ogawa, A.~Sakai, \emph{Localization of extriangulated categories}.
J. Algebra \textbf{611} (2022), 341--398.

\bibitem[NP]{NP19}
H.~Nakaoka, Y.~Palu, \emph{Extriangulated categories, Hovey twin cotorsion pairs and model structures}.
Cah. Topol. G\'{e}om. Diff\'{e}r. Cat\'{e}g. \textbf{60} (2019), no. 2, 117--193.

\bibitem[Nee]{Nee01}
A.~Neeman, \emph{Triangulated categories}.
Annals of Mathematics Studies 148, Princeton University Press (2001).

\bibitem[Oga1]{Oga21}
Y.~Ogawa, \emph{Auslander's defects over extriangulated categories: an application for the general heart construction}.
J. Math. Soc. Japan \textbf{73} (2021), no. 4, 1063--1089.

\bibitem[Oga2]{Oga22}
Y.~Ogawa, \emph{Abelian Categories from Triangulated Categories via Nakaoka--Palu's Localization}.
Appl. Categ. Structures \textbf{30} (2022), no. 4, 611--639.

\bibitem[PPPP]{PPPP19}
A.~Padrol, Y.~Palu, V.~Pilaud, P-G.~Plamondon, \emph{Associahedra for finite-type cluster algebras and minimal relations between $\mathbf{g}$-vectors}.
Proc. Lond. Math. Soc. (3) {\bf 127} (2023), no. 3, 513--588.

\bibitem[Pal]{Pal08}
Y.~Palu, \emph{Cluster characters for 2-Calabi-Yau triangulated categories}.
Ann. Inst. Fourier (Grenoble) \textbf{58} (2008), no. 6, 2221--2248.

\bibitem[Rum1]{Rum01}
W.~Rump, \emph{Almost abelian categories}.
Topologie G\'eom. Diff\'erentielle Cat\'eg. \textbf{42} (2001), no. 3, 163--225.

\bibitem[Rum2]{Rum21}
W.~Rump, \emph{The acyclic closure of an exact category and its triangulation}.
J. Algebra \textbf{565} (2021), 402--440.

\bibitem[Sak]{Sak21}
A.~Sakai, \emph{Relative extriangulated categories arising from half exact functors}.
J. Algebra \textbf{614} (2023), 592--610.

\bibitem[Sal]{Sal79}
L.~Salce, \emph{Cotorsion theories for abelian groups}
in Symposia Mathematica,
Vol. XXIII (Conf. Abelian Groups and their Relationship
to the Theory of Modules, INDAM, Rome, 1977) (1979), 11--32.

\bibitem[Tat]{Tat21}
A.~Tattar, \emph{The Structure of Aisles and Co-aisles of t-Structures and Co-t-structures}.
Appl. Categ. Structures {\bf 32} (2024), no. 1, Paper No. 5.

\bibitem[Ver]{Ver96}
J.-L. Verdier,
\newblock {\em Des cat{\'{e}}gories d{\'{e}}riv{\'{e}}es des cat{\'{e}}gories
  ab{\'{e}}liennes}.
\newblock Ast{\'{e}}risque, {\bf 239}:xii+253 (1997), 1996.
\newblock With a preface by Luc Illusie, Edited and with a note by Georges
  Maltsiniotis.

\end{thebibliography}
\end{document}